\documentclass[10pt]{amsart}
\usepackage{amsbsy}
\usepackage{graphicx,epsfig,subfigure,psfrag}
\usepackage{color}
\begin{document}

\newtheorem{theorem}{Theorem}
\newtheorem{proposition}{Proposition}
\newtheorem{lemma}{Lemma}
\newtheorem{corollary}{Corollary}
\newtheorem{definition}{Definition}
\newtheorem{remark}{Remark}
\newcommand{\tex}{\textstyle}
\numberwithin{equation}{section} \numberwithin{theorem}{section}
\numberwithin{proposition}{section} \numberwithin{lemma}{section}
\numberwithin{corollary}{section}
\numberwithin{definition}{section} \numberwithin{remark}{section}
\newcommand{\ren}{\mathbb{R}^N}
\newcommand{\re}{\mathbb{R}}
\newcommand{\n}{\nabla}
\newcommand{\p}{\partial}
\newcommand{\iy}{\infty}
\newcommand{\pa}{\partial}
\newcommand{\fp}{\noindent}
\newcommand{\ms}{\medskip\vskip-.1cm}
\newcommand{\mpb}{\medskip}
\newcommand{\AAA}{{\bf A}}
\newcommand{\BB}{{\bf B}}
\newcommand{\CC}{{\bf C}}
\newcommand{\DD}{{\bf D}}
\newcommand{\EE}{{\bf E}}
\newcommand{\FF}{{\bf F}}
\newcommand{\GG}{{\bf G}}
\newcommand{\oo}{{\mathbf \omega}}
\newcommand{\Am}{{\bf A}_{2m}}
\newcommand{\CCC}{{\mathbf  C}}
\newcommand{\II}{{\mathrm{Im}}\,}
\newcommand{\RR}{{\mathrm{Re}}\,}
\newcommand{\eee}{{\mathrm  e}}
\newcommand{\LL}{L^2_\rho(\ren)}
\newcommand{\LLL}{L^2_{\rho^*}(\ren)}
\renewcommand{\a}{\alpha}
\renewcommand{\b}{\beta}
\newcommand{\g}{\gamma}
\newcommand{\G}{\Gamma}
\renewcommand{\d}{\delta}
\newcommand{\D}{\Delta}
\newcommand{\e}{\varepsilon}
\newcommand{\var}{\varphi}
\newcommand{\lll}{\l}
\renewcommand{\l}{\lambda}
\renewcommand{\o}{\omega}
\renewcommand{\O}{\Omega}
\newcommand{\s}{\sigma}
\renewcommand{\t}{\tau}
\renewcommand{\th}{\theta}
\newcommand{\z}{\zeta}
\newcommand{\wx}{\widetilde x}
\newcommand{\wt}{\widetilde t}
\newcommand{\noi}{\noindent}
\newcommand{\uu}{{\bf u}}
\newcommand{\xx}{{\bf x}}
\newcommand{\yy}{{\bf y}}
\newcommand{\zz}{{\bf z}}
\newcommand{\aaa}{{\bf a}}
\newcommand{\cc}{{\bf c}}
\newcommand{\jj}{{\bf j}}
\newcommand{\ggg}{{\bf g}}
\newcommand{\UU}{{\bf U}}
\newcommand{\YY}{{\bf Y}}
\newcommand{\HH}{{\bf H}}
\newcommand{\GGG}{{\bf G}}
\newcommand{\VV}{{\bf V}}
\newcommand{\ww}{{\bf w}}
\newcommand{\vv}{{\bf v}}
\newcommand{\hh}{{\bf h}}
\newcommand{\di}{{\rm div}\,}
\newcommand{\ii}{{\rm i}\,}
\def\I{{\rm Id}}
\newcommand{\inA}{\quad \mbox{in} \quad \ren \times \re_+}
\newcommand{\inB}{\quad \mbox{in} \quad}
\newcommand{\inC}{\quad \mbox{in} \quad \re \times \re_+}
\newcommand{\inD}{\quad \mbox{in} \quad \re}
\newcommand{\forA}{\quad \mbox{for} \quad}
\newcommand{\whereA}{,\quad \mbox{where} \quad}
\newcommand{\asA}{\quad \mbox{as} \quad}
\newcommand{\andA}{\quad \mbox{and} \quad}
\newcommand{\withA}{,\quad \mbox{with} \quad}
\newcommand{\orA}{,\quad \mbox{or} \quad}
\newcommand{\atA}{\quad \mbox{at} \quad}
\newcommand{\onA}{\quad \mbox{on} \quad}
\newcommand{\ef}{\eqref}
\newcommand{\mc}{\mathcal}
\newcommand{\mf}{\mathfrak}

\newcommand{\ssk}{\smallskip}
\newcommand{\LongA}{\quad \Longrightarrow \quad}
\def\com#1{\fbox{\parbox{6in}{\texttt{#1}}}}
\def\N{{\mathbb N}}
\def\A{{\cal A}}
\newcommand{\de}{\,d}
\newcommand{\eps}{\varepsilon}
\newcommand{\be}{\begin{equation}}
\newcommand{\ee}{\end{equation}}
\newcommand{\spt}{{\mbox spt}}
\newcommand{\ind}{{\mbox ind}}
\newcommand{\supp}{{\mbox supp}}
\newcommand{\dip}{\displaystyle}
\newcommand{\prt}{\partial}
\renewcommand{\theequation}{\thesection.\arabic{equation}}
\renewcommand{\baselinestretch}{1.1}
\newcommand{\Dm}{(-\D)^m}

\title
{\bf Blow-up scaling and global behaviour\\ of  solutions of the
bi-Laplace equation\\ via pencil operators}

\author{Pablo~\'Alvarez-Caudevilla and Victor~A.~Galaktionov}

\address{Universidad Carlos III de Madrid,
Av. Universidad 30, 28911-Legan\'es, Spain -- Work phone number: +34-916249099}
\email{pacaudev@math.uc3m.es}

\address{Department of Mathematical Sciences, University of Bath,
 Bath BA2 7AY, UK}
\email{vag@maths.bath.ac.uk}

\keywords{Bi- and Laplace equations, higher-order equations,
pencil of non self-adjoint operators,  harmonic polynomials, nodal sets}

\thanks{This work has been partially supported by the Ministry of Economy and Competitiveness of
Spain under research project MTM2012-33258.}

 \subjclass{31A30, 35A20, 35C11, 35G15}
\date{\today}




\begin{abstract}
 As the main problem, the bi-Laplace equation
 \[
  \D^2 u=0 \quad (\D=D_x^2+D_y^2)
 \]
 in a bounded domain $\O \subset \re^2$,
 with inhomogeneous Dirichlet
 or Navier-type conditions on the smooth boundary $\p \O$ is
 considered. In addition, there is a finite collection of  curves
  \[
 \Gamma = \Gamma_1\cup...\cup\Gamma_m \subset \O, \quad 
 \mbox{on which we assume homogeneous Dirichlet 
 conditions} \quad u=0,
  \]
  focusing at the origin $0 \in \O$ (the analysis would be similar for any other point) .
  This  makes the above elliptic problem overdetermined. Possible
  types of the behaviour of solution $u(x,y)$ at the tip $0$ of such
  admissible multiple
  cracks, being a singularity point, are described, on the basis of
   blow-up
  scaling techniques and spectral theory of pencils of non
  self-adjoint operators. Typical types of admissible cracks are
  shown to be governed by nodal sets of a countable family of
  {\em
harmonic
 polynomials}, which are now represented as  pencil eigenfunctions, instead of their
 classical representation via a standard Sturm--Liouville
 problem. Eventually, for a fixed admissible crack formation at
 the origin, this allows us to describe {\em all} boundary data,
 which can generate such a blow-up crack structure. In particular,
 it is shown how the co-dimension of this data set increases with
 the number of asymptotically straight-line cracks focusing at 0.

\end{abstract}

\maketitle

\section{Introduction}
 \label{S1}

\subsection{Models and preliminaries}


\noindent   In this work we intend to ascertain the behaviour
of the solutions of the \emph{bi-Laplace equation} with Dirichlet
boundary conditions in a bounded smooth domain $\O \subset \re^2$
\begin{equation}
\label{bila}
\left\{\begin{array}{cc}
    \Delta^2 u =0   &  \hbox{in}\quad \O,\\
    u=f(x,y) &  \hbox{on}\quad \G,\\
    u=g(x,y),\;\; \frac{\p u}{\p {\bf n}}=h(x,y) &  \hbox{on}\quad \p\O,\\
\end{array} \right.
\end{equation}
where $\D=D_x^2+D_y^2$ is the standard Laplace operator in $\re^2$, ${\bf n}$ stands for the unit outward
normal to $\p\O$, and $f$, $g$ and $h$ are given smooth functions in
$\O$, so that $g^2(x,y)+h^2(x,y) \not \equiv 0$. In our particular case, $\O$
is assumed to have, what we refer to as a multiple crack $\G$ which is composed of a finite collection
of $m \ge 1$ curves (to be described below)
 \be
 \label{curv1}
  \G= \G_1 \cup \G_2 \cup...\cup \G_m \subset \O  \quad \mbox{such that each $\G_j$ passes through the
  origin $0 \in \O$}.
   \ee
The origin is then  the tip of this crack. 
Indeed, in the present research, we assume that, near the origin, in the
lower half-plane $\{y<0\}$ (similarly in the upper half-plane $\{y>0\}$), all cracks asymptotically  take a
straight line form, i.e. as shown in Figure \ref{figcrack},
 \be
 \label{str1}
 \G_k:  \,\, x=\a_k (-y)(1+o(1)), \,\, y \to 0, \,\,\,
 k=1,2,...,m,\,\,
 \mbox{where} \,\,\,
 \a_1<\a_2<...< \a_m
  \ee
  are given constants. Basically we are choosing cracks of the type described by \eqref{str1} that 
  will allow us to obtain possible types of the behaviour of solutions $u(x,y)$ of the equation \eqref{bila}.  
  Indeed, \emph{a posteriori} the analysis carried out throughout this paper will show that those types of cracks are the only
  admissible ones. 
  For further extensions a different analysis must be done.  
Thus, the precise statement of the problem assumes that geometrical conditions such as \ef{str1}
describe {\em all the admissible cracks near the origin}, i.e.
{\em no other straight-line cracks are considered}.

Moreover, in our basic model, we assume homogeneous Dirichlet
conditions on the crack for the bi-Laplacian problem \eqref{bila}
\be
 \label{Dir1}
 u=0 \onA \G,
  \ee
   that makes the problem overdetermined, so that only some types of such
   multiple cracks \ef{curv1}, \ef{str1} are admissible. Note that, since we will obtain an explicit expression 
   for the solutions of the bi-Laplace problem \eqref{bila}, \emph{a posteriori} we 
   will be able to impose any condition on the crack $\Gamma$. 
   
   Thus, our main goal is to describe all possible types of admissible
   multiple cracks, for which the above elliptic problem can have
   a solution, at least for some boundary data $g$ and $h$ in
   \ef{bila}. To this end, actually, we need to describe all
   types of zero or nodal sets, which are admitted by oscillatory solutions $u(x,y)$ of the bi-Laplace
   equation.
   
   Therefore, performing a proper rescaling and using non-self-adjoint spectral pencil operator 
   theory (see  \cite{Kon1,Kon2,KL} and Section\;\ref{S2} in this paper for further details about these types of operators),
    we are able to show
special linear combinations of  ``harmonic polynomials" ascertaining important qualitative
information about the behaviour
  of the solution based on the nodal set of harmonic polynomials, especially close to the tip of the crack
  $\G$ for which we will describe all the admissible types of cracks.
Specifically, we show that their nodal sets
 play a key role in the general multiple crack problem for various
 equations. In fact, the analysis is extended, as an example to future improvements, to a couple of 
  non-linear problems, and intends to provide an alternative methodology in the analysis of similar problems with singularity points on the boundary.

 Consequently, the novelty of this work consists in approaching the study of some boundary value problems,
exhibiting a crack singularity, by a ``blow-up" and ``elliptic evolution" approach
  using the spectral theory of pencils of non self-adjoint
operators.

In fact we base our analysis on 
the application of the
  spectral theory of pencil operators, transforming the problem appropriately and, hence, reducing it to solve 
  a 1D spectral problem. We also believe that the analysis presented here can be extended to 
  other problems providing a different technique to obtain important qualitative information; 
  see for instance \cite{CGpLap} for an application of the techniques presented here for a $p$-Laplacian problem.

 Note also that, nowadays, this kind of technique is normally used
for parabolic or hyperbolic equations, not elliptic (cf.
\cite{AdPiv,GalHardy}). On the other hand, it is necessary to
recall that the pioneering Kondratiev's study in the 1960s \cite{Kon1,
Kon2} of boundary regularity/asymptotics for general linear
elliptic (and ultra-parabolic) equations was, for the first time,
performed via an ``elliptic evolution approach", with all typical
features available: suitable blow-up scaling at a boundary point,
operator pencil analysis, etc. 

In particular, Kondratiev applied those blow-up scalings to 
analyse non-smooth domains, such as domains with corner points, edges, etc on the boundary, forming cone-type domains (boundary shape like a cone). 
After some variable transformations he was able to transform the equation 
into a pencil spectral problem permitting the analysis at those corner points. 

In this work, we have just used the fundamental ideas of Kondratiev, performing a different change of variable and, hence,
obtaining a different pencil of non-self-adjoint operator that allows us to ascertain the behaviour of the solutions for that multiple crack section \eqref{str1}. 
However, our cone is in the interior and, actually with no connection with the boundary (the tip of the crack is not on the boundary) 
while Kondratiev analysed problems in cone-type domains on the boundary. 

Thus, as we will show below we obtained completely different functional spaces and solutions with a specific polynomial form which provide us with the behaviour 
at the tip of the cracks.  Indeed, we will show that this {\em
internal crack problem} requires polynomial eigenfunctions of
different pencils of linear operators, which were not under
scrutiny in Kondratiev's works.

Those arguments were used in  \cite{EGKP}  for a case where the boundary, 
after blow-up scalings, is at infinity, and it was not easy to convince Kondratiev himself that it was a new and different case.

\begin{figure}[htp]

\includegraphics[scale=0.4]{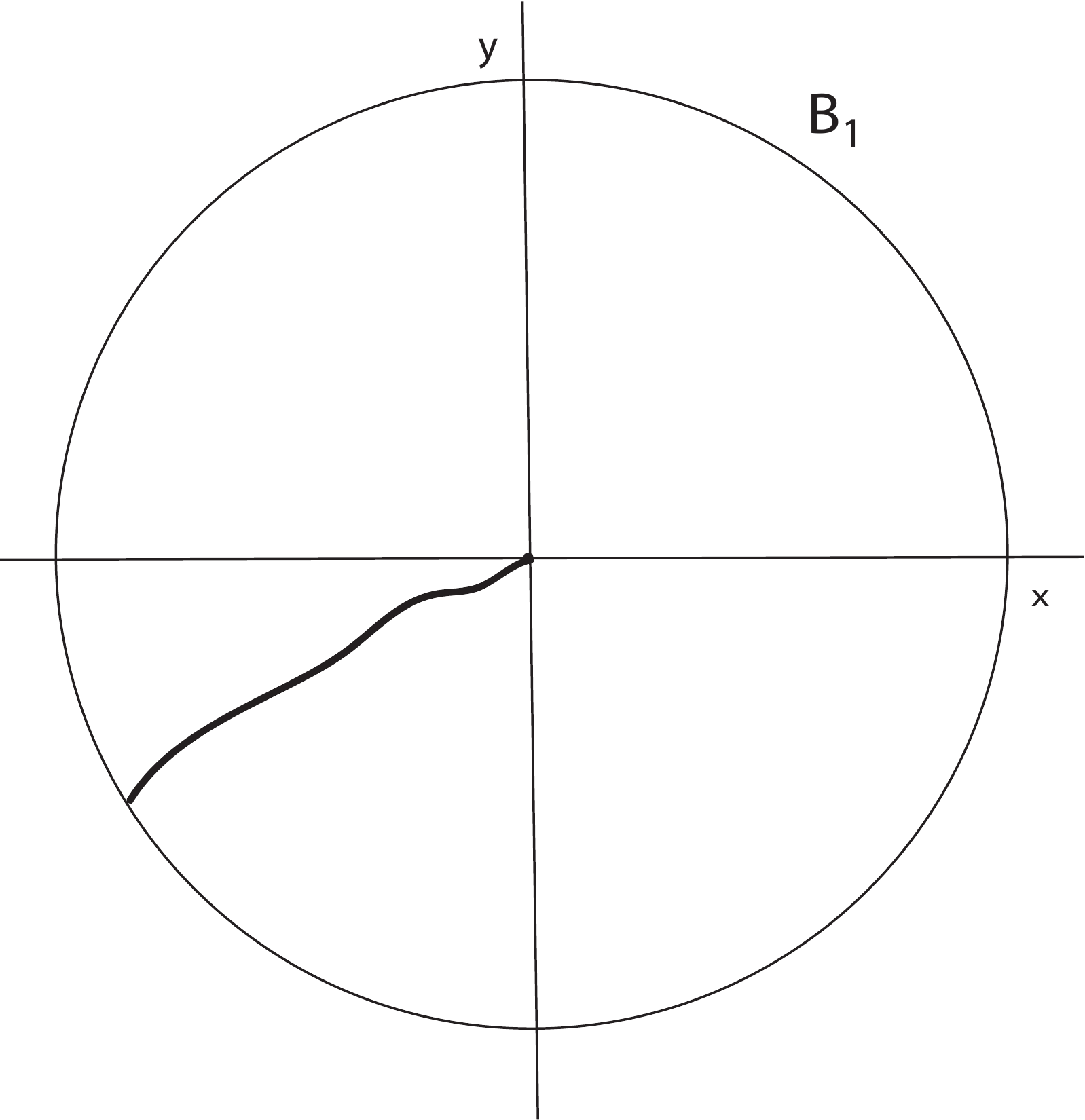}

\vskip -0.2cm \caption{ \small One-crack model.} \label{figcrack}
\end{figure}


\subsection{Approach and main results}


In the study of such admissible cracks, i.e. the behaviour of the
solutions when $(x,y) \to (0,0)$, it suffices to consider
 \be
 \label{dom}
\mc{D}= B_1\setminus \G, \quad \mbox{a
unit ball in $\re^2$ centered at the origin $0$ minus the crack $\G$}.
 \ee
 For other, not pointwise blow-up
estimates, we continue to consider general smooth domains $\O$.

Thus, even though our main motivation to develop this work was the
analysis of the bi-Laplace equation \eqref{bila}, we shall start
with similar multiple crack issues for the Laplacian.
Since the bi-Laplace operator is the iteration of two
  Laplacians, inevitably,
   we will need to start the analysis of the problem by using the pure single Laplacian
   \be
 \label{Lap1}
  \D u=0, \inB \O \,\,(=B_1), \quad u=f\, (\not \equiv 0) \onA \p \O, \quad
   u=0 \onA \G,
   \ee
   to obtain those results. Note that here we have shifted $f$ on the boundary with respect to the problem \eqref{bila}.

 Additionally, we shall complete our work with the study
  of several other problems as well. These problems are going to be defined
again  under the geometrical condition \eqref{str1} but
considering various different operators, such as other semilinear
 related equations.

\vspace{0.2cm}

\noindent\underline{\em First step. Laplace equation with multiple
cracks}.
 As a
by-product of our approach, we
consider the problem for the Laplace equation \eqref{Lap1}. 
 For this simpler problem in Section\;\ref{SLap},
 we prove that all the solutions with cracks at 0 must satisfy
 \be
 \label{decola}
  \tex{
  u(x,y)=w(z,\t)= \sum_{(k \ge l)} \eee^{-k \t} [c_k\psi_{k,1}^*(z)+ d_k \psi_{k-1,2}^*(z)],
   \quad \mbox{with} \quad
c_l^2+d_l^2 \ne 0,
  }
  \ee
  where
  \be
  \label{scal}
  \tex{  z= x/(-y) \quad \hbox{and}\quad  \t = - \ln
(-y) \forA y<0,
}
\ee
 where $\{\psi_{k,1}^*(z),\psi_{k-1,2}^*(z)\}$ are two families of harmonic polynomials (re-written in terms of the rescaled variable $z$),
   denoted by
 \[
 \tex{\psi_{l_+}^*(z)\equiv \psi_{l,1}^*(z) \quad \hbox{and} \quad \psi_{l_-}^*(z)\equiv \psi_{l-1,2}^*(z), \quad \hbox{for any}\quad  l=m,m+1,\cdots}
 \]
 associated with two families of eigenvalues
 \[
 \tex{
  \l_l^+=-l, \quad l=1,2,3,... \andA \l_l^-=-l-1, \quad
  l=0,1,2,3,... \,.
  }
\]
 and, such that, they have a polynomial expression 
 \[
 \tex{\psi_l^*(z)=\sum\limits_{k=l,l-2,...,0} a_k z^k\quad (a_l=1).}
 \]
 Rescaling \eqref{scal} corresponds to a blow-up scaling near the origin, moving the singularity point at the origin into an asymptotic convergence
 when $\tau$ goes to infinity. 
 Moreover, we will keep this ``blow-up scaling logic" for the rest of other
similar problems to appear. Also, this scaling approach could be extended to $y>0$ in a similar way.

Note that we choose such an expression depending on $z$ for convenience since, as it is shown later in Section\;\ref{SLap}, 
we deal with eigenfunctions of a quadratic pencil of operators
 and not with a standard Sturm--Liouville problem. 
 
 Indeed, after performing the rescaling \eqref{scal} (and then the separation variables method) 
we transform the Laplace equation into a pencil of non-self-adjoint operator, in particular for this 
case, a quadratic pencil operator of the form
$$(1+z^2)(\psi^*)''+2(\l+1)z (\psi^*)' +\l(\l+1)
 \psi^*=0.
$$   
Therefore, in the main result of Section\;\ref{SLap}, Theorem\;\ref{Th.2}, we prove 
that the coefficients $c_l, \, d_l \in \re$ in \eqref{decola} are arbitrary constants that satisfy 
\[c_l^2+d_l^2 \ne 0.\] 
Indeed, we observe in  
the first leading terms while approaching the origin, a linear
combination of those two families of eigenfunctions  as classic harmonic
polynomials.

On the other hand, it is also proved that, if all $\{\a_k\}$ in \ef{str1} do not coincide with
all $m$ subsequent zeros of any nontrivial linear combination
 \be
 \label{dm1}
c_l \psi_{l,1}^*(z)+d_l \psi^*_{l-1,2}(z), \quad \mbox{with} \quad
c_l^2+d_l^2 \ne 0,
 \ee
 then the multiple crack problem \ef{Lap1}
cannot have a solution for any boundary Dirichlet data $f$ on
$\p \O$.

However, a solution exists if for some $l$ the zero condition is satisfied, i.e. $\a_l$ coincides with the zeros 
of the linear combination of harmonic polynomials \eqref{dm1},  and 
\[
\tex{|u(x,y)|=O(|(x,y)|^l) \quad \hbox{as}\quad (x,y)\to (0,0).}
\]
Moreover and obviously, restricting to $\G$
all types of
admissible crack-containing expansions \ef{decola} (with closure
in any appropriate functional space), {\em fully describes all
types of boundary data, which lead to the desired crack formation
at the origin a posteriori}.
The previous discussion is summarised in Theorem \ref{Th.2}.

Although, one can assume the function $u(x,y)$ in \eqref{decola} belonging to the Sobolev space $W^{1,2}(\O\setminus\G)=H^1(\O\setminus\G)$, 
 due to the expansions considered (for the Laplace problem \eqref{Lap1} and also for the bi-Laplace \eqref{bila} $W^{2,2}$) in our analysis we actually have that
 the eigenfunctions are harmonic functions of Hermite-type polynomials which are complete in any appropriate 
 $H^1_\rho$ or $L^p_\rho$-space, where the weight $\rho$ has a exponential decay at infinite. For example 
 $$\rho(z) \sim {\mathrm e}^{- a z^2}\quad \hbox{(or ${\mathrm e}^{-a |z|}$)},\quad \hbox{$a>0$ small,}$$
would be enough;  see \cite{KolF} for further details about this classical analysis.

\begin{remark}. 
Furthermore, it follows from \ef{dm1} that any admissible crack distribution
governed by zeros of the polynomial \ef{dm1}, i.e. the nodal set of the polynomial eigenfunctions, represented by pencil eigenfunctions instead of the 
classical one from the Sturm--Liouville problem, for any $l=1,2,...$,
contains a single free parameter (say, $\frac {d_l}{c_l} \in \re$,
$c_l \ne 0$). 

In other words, the whole set of admissible multiple
straight-line crack formations \ef{str1} (and basically the reason we assume such a family of cracks) comprises no more than a
{\em countable family of one-dimensional subsets}\footnote{Bearing
in mind the rotational invariance of the Laplace operator
 (as a one-dimensional group of orthogonal transformations in
 $\re^2$), the total family of {\em essentially distinct} admissible crack configurations
becomes no more than a {\em countable subset}, which is described
by subsequent zeros of the harmonic polynomials.
At the same time, general Dirichlet data $f(x,y)$ on $\O\setminus\G$ is
obviously characterised as an uncountable subset.
}
 (recall that
this is true for arbitrary Dirichlet data $f(x)$ on $\p\O$, which,
as we mentioned above, {\em a posteriori}, can be completely
described).

\end{remark}

\noindent\underline{\em Extensions to semi-linear equations}. Although not the purpose of this paper, using a couple of examples 
\be
\label{semil22}
\tex{\D u + |u|^{p-1}u=0\quad \hbox{and} \quad \D u + \frac{|u|^{p-1}u}{x^2+y^2}=0, \inB \O \subset \re^2 \whereA p>1},
\ee
via similar scalings and asymptotic analysis, we are able to show the types of decay patterns at the origin showing 
an application of the results obtained for the Laplace problem \eqref{Lap1} to semi-linear equations. We just use those two 
equations as examples of what could happen for non-linear problems.

Indeed, for the first equation 
we find that the nonlinear term is negligible, i.e. it cannot affect those patterns. 
We will show that performing the same rescaling \eqref{scal} as for the Laplace equation \eqref{Lap1} 
will lead to an exponentially small perturbation in $\tau$ of the rescaled Laplacian one. Hence, when $\tau$ goes to infinity (this means, due to the rescaling, close to the origin
for the variables $(x,y)$)
the nonlinearity does not have any effect for the decay patterns at infinity (or the origin).

However, for the second equation, for which we involve the 
nonlinearity in a formation of multiple zeros at the origin we show through some numerical analysis, the nodal sets of several nonlinear eigenfunctions
obtained after the rescaling.   

Note that we just use those two equations as examples for the asymptotic behaviour at infinity, after the rescaling. For those purposes 
we only need to consider that $p>1$.  

\ssk

\noindent\underline{\em Finally, the bi-Laplace equation with
multiple cracks}. Obviously, since $\D^2=\D \D$, the solutions of
the Laplace equation also solve the bi-Laplacian. Therefore, some
of the results obtained for the Laplacian (see Section\;\ref{SLap}
for further details) can be translated and applied to \ef{bila}
with the same crack constraints, though, nevertheless, the latter
one is more demanding. Indeed, a full description of admissible
multiple crack configurations \ef{str1} for \ef{bila} is more
difficult.

For the problem \eqref{bila} we find that the solutions have an expression of the form
\be
\label{exp44}
  \tex{
  u(x,y)=w(z,\t)= \sum_{(k \ge l)} \eee^{-k \t} [C_k\psi_{k,1}^*(z)+ D_k \psi_{k-1,2}^*(z)+E_k \psi_{k-2,3}^*(z)+F_k \psi_{k-3,4}^*(z)],
  }
 \ee
where, again, we use the same scaling \eqref{scal} and
 \[
 \phi^*=\{\psi_{l,1}^*,\psi_{l,2}^*, \psi_{l,3}^*,\psi_{l,4}^* \},
 \]
are four harmonic polynomial eigenfunction families, with the same
$z$-representation, which are complete in any reasonable weighted
$L^2$ space such that
\[
 \psi_{l,1}^*(z) \equiv \psi_{l,1}^*(z),\quad \psi_{l,2}^*(z) \equiv \psi_{l-1,2}^*(z),\quad \psi_{l,3}^*(z)
  \equiv \psi_{l-2,3}^*(z),\quad \psi_{l,4}^*(z) \equiv \psi_{l-3,4}^*(z),
  \]
  now associated with four families of negative eigenvalues
   \begin{align*} & \tex{
  \l_{l,1}=-l, \quad l=1,2,3,... \qquad \quad \l_{l,2}=-l-1, \quad
  l=0,1,2,3,...} \\ &  \tex{\l_{l,3}=-l-2, \quad
  l=0,1,2,3,... \andA  \l_{l,4}=-l-3, \quad
  l=0,1,2,3,... \,
  }
  \end{align*}
  of the corresponding pencil operator
  \begin{align*}
 {\bf F}_\l^* \psi^* \equiv  & \{(\l^4+6\l^3+11\l^2+6\l)I+ 4 (\l^3+6\l^2+11\l) zD_z
 \\ & +2(1+3z^2)(\l^2+5\l)D_z^2+4\l(1+z^2)z D_z^3  - {\bf C}^*\} \psi^*=0.
 \end{align*}  
  Note that again we use  such
non-standard notations of harmonic polynomials in order to fit our
operator pencil approach.

Thus, four collections of expansion coefficients $\{C_k\}$,
$\{D_k\}$, $\{E_k\}$, and $\{F_k\}$
  (which depend on boundary
 data on $\p\O$ admitted all types of cracks at 0, which is obviously described via \ef{exp44}) take place so that
 if all $\{\a_k\}$ of the multiple cracks \eqref{str1} do not coincide with  all $m$ subsequent zeros of any
nontrivial linear combination
\[
C_l \psi^*_{l,1}(z)+D_l \psi^*_{l,2}(z)+E_l \psi^*_{l,3}(z)+F_l \psi^*_{l,4}(z), \quad \mbox{with} \quad
C_l^2+D_l^2+ E_l^2+F_l^2 \ne 0,
\]
 then the multiple crack problem \ef{bila}
cannot have a solution for any boundary Dirichlet data $g$, $h$ on
$\p\O$ for the problem \eqref{bila}. 

\begin{remark}. 
The linear combinations previously shown arise naturally from the spectral theory of the operators involved (Laplacian, bi-Laplacian). Indeed, for the Laplacian since
$u$ is harmonic it can be decomposed as a sum of homogeneous harmonic functions. In particular, for the Laplacian operators we obtain two families of eigenfunctions 
denoted by
$$\{\psi^*_{k,1}\},\quad \{\psi^*_{k-1,2}\}.$$
The difficulty in ascertaining here the results at the tip of the cracks comes from the regularity problem we are facing here in $\O=B_1$,
for the eigenvalue problem, since we have a singularity point
at the origin. However, we are in the context analysed in \cite{AL} so that
 \[
 \tex{w(z , \t)= {\mathrm e}^ {\l \t} \psi^*(z ),\quad \hbox{where} \quad  {\rm Re} \, \l<0,}
 \]
 and for a orthonormal   basis $\{\psi^*_{k,1},\psi^*_{k-1,2}\}$ of harmonic polynomial eigenfunctions,
  we find that the solutions of the problem \eqref{Lap1}
 are a decomposition of the form \eqref{decola}. We conclude similar
  arguments for the bi-Laplacian problem \eqref{bila}.
  \end{remark}
  
 \begin{remark}
 We would like to explain that the rescaling introduced in this paper \eqref{scal} allows us to get solutions 
 of the polynomial form \eqref{decola} and \eqref{exp44} as special linear combinations of  ``harmonic polynomials" that show the behaviour of the problems considered here at the tip of the cracks
 depending on the nodal set of those ``harmonic polynomials".  Indeed, our rescaling \eqref{scal} since it is different from the one used by Kondratiev in \cite{Kon1,Kon2} 
 produces a different pencil operator and, hence, different eigenfunctions for the corresponding pencil non-self-adjoint operators. However, the expansions 
 \eqref{decola} and \eqref{exp44} based on those polynomial eigenfunctions are capable of providing the behaviour of the solutions at the tip of the cracks. 
   \end{remark}

 \begin{remark}
As one of the referees of this paper pointed out for the problems under analysis in this paper (Laplace, class of perturbations of the Laplace and Bi-Laplace) solutions 
with a ``strong zero" ($(|u(x)| = O(|x|^N)$, for any $N > 0$) do not exist. Indeed, results
on Unique Continuation guarantee that any solution with a strong zero is
identically zero.  
\end{remark}

\vspace{0.2cm}

\noindent\underline{\em Some further extensions}.  
Though
our approach is done in two dimensions, the scaling blow-up
approach applies to $\O$ in $\re^3$ (or any $\ren$), where
spherical polynomials naturally occur such that their nodal sets
(finite combination of nodal surfaces) of their  linear
combinations, as above, describe all possible local structures of
cracks concentrating at the origin. However, if $N>2$ the possible geometry of the 
crack $\G$ is far richer.

 Additionally, though very difficult to achieve, one can study similar crack problems for other
 elliptic equations such as
  \[
  u_{xxxx}+u_{yyyy}=0,
  \]
  where non-standard harmonic-like polynomials naturally occur, such as
  eigenfunctions of a {\em quartic operator pencil}, like the ones we obtained for the bi-Laplace equation \eqref{bila}.

Further extensions to quasilinear
 equations, but out of the scope of this paper, 
  such as the
quasilinear {\em $p$-Laplace equation} 
\be
 \label{p1}
 \D_p u \equiv \n \cdot (|\n u|^{p-2} \n u) =0,
  \ee
could be also carried out applying this analysis; see \cite{CGpLap}. However, in this particular case 
after performing a similar blow-up scaling one arrives at a nonlinear eigenvalue problem. Since 
this problem is nonlinear and the ideas used
for the linear case, when $p=2$, cannot be applied directly,
the corresponding {\em nonlinear
  eigenfunctions} of the nonlinear pencil for \ef{p1} should be
  obtained by branching from harmonic polynomials as
  eigenfunctions of the quadratic pencil that occurs for the
  Laplacian in \ef{Lap1}.
  
  The main reason is due to the fact that for the Laplace equation \eqref{Lap1}, as explained above, it is possible to explicitly obtain
two families of negative eigenvalues associated, respectively, to two families of eigenfunctions.
Indeed, from the expressions of the two families of eigenvalues 
and using the pencil operator theory, we can ascertain the coefficients of every eigenfunction
explicitly as well.
However, following the same argument for the nonlinear PDE \eqref{p1} it is not possible, in general,
to get the corresponding families of eigenvalues and, hence, the associated eigenfunctions. Therefore, the branching argument 
  used in \cite{CGpLap} is applied to get such information about the eigenfunctions, and hence eigenvalues of a nonlinear problem.

\section{Pencils of linear operators: preliminaries}
 \label{S2}


\noindent As one of the main tools we are using in this work to get to the results 
and in a direct connection with our
blow-up evolution approach we introduce the Theory of Pencil Operators.  
Let us mention that pencil operator
theory appeared and was crucially used in the regularity and
asymptotic analysis of elliptic problems in a seminal paper by
Kondratiev \cite{Kon2} and also for parabolic problems in
\cite{Kon1}, where spectral problems, that are nonlinear
(polynomial) in the spectral parameter $\l$, occurred. Later on,
Mark Krein and Heinz Langer \cite{KL} made a fundamental
contribution to this theory analysing the spectral theory for
strongly damped quadratic operator pencils. In general, a polynomial pencil
operator is denoted by
\be
\label{penop}
 A(\l):=A_0+\l A_1+\cdots+\l^n A_n,
 \ee
 where $\l \in
{\mathbb C}$ is a spectral parameter and $A_i$, with
$i=0,1,\cdots, n$, are linear operators acting on a Hilbert space
$X$ (here we might assume for example that $X$ can be $H^1_\rho$ or $L^2_\rho$ with any reasonable weight $\rho$). 
Operators of the form \eqref{penop} are sometimes called Polynomial matrix when the linear differential operators
$A_i$ are matrices. A linear pencil of operators has the form
  \[
  A(\l):= A-\l B,
  \]
where $A,B$ are two linear operators. In the simplest case, we
have the linear pencil operator
 \[
 A(\l)= A-\l \, {\rm Id}, \quad
\hbox{or}\quad A(\l)={\rm Id} -\l A,
 \]
  which represents the usual (standard)
linear spectral  problems. A clear difference between those
spectral linear problems and the pencil operators is essentially
that, for the simplest pencil operators, the set of eigenvalues is
obtained as the roots of the characteristic equation
 \[
 {\rm
det}\, A(\l)=0,
   \]
    i.e. powers of the values $\l_k$, with the basis
of the eigenspace as
\[\{\psi_k,\l_k\psi_k,\cdots,\l_k^{n-1}\psi_k\}.\]

Furthermore, the analysis  of polynomial pencil operators has been under
scrutiny for many years in order to study spectral problems of the
form \eqref{penop} and, as pointed out by Markus \cite{Mar}, arise
naturally in diverse areas of mathematical physics (differential
equations and boundary value problems), with applications to
Elasticity, Hydrodynamics problems, among other things. In the
pioneering work of M.V.~Keldysh in 1951 (earlier first ever
results of J.D.~Tamarkin's  PhD Thesis of the 1917 should be
mentioned as well; see Markus \cite{Mar} for this amazing part of
the history of
       mathematics)
pencils, including multiplicity results and completeness of the
set of eigenvectors, even for non-self-adjoint operators, were
thoroughly analysed.

As mentioned at the beginning of this
section, one of the most important contributions, and related with
the analysis carried out here, was made by Krein \& Langer
\cite{KL} who developed further approaches for quadratic pencil
operators of the form 
\[A(\l)=\l^2 A_2+\l A_1+A_0.\]
  In this paper,
we will use elements  of this well-developed spectral theory of
non-self-adjoint quadratic or of fourth order pencil polynomials,
though not that profound ones, since, for linear elliptic
problems, we always deal with polynomial (harmonic)
eigenfunctions, which cause no problem concerning their
completeness, closure, and further functional properties.

\section{The Laplace equation: crack distribution via nodal sets of transformed harmonic polynomials}
 \label{SLap}

Since $\D^2=\D \D$, we inevitably should first consider the
multiple crack problem for the Laplace equation \ef{Lap1}.
Obviously, these admissible crack distributions remain valid for
the bi-Laplace \ef{bila}, \ef{Dir1}, but, in addition, there are
other types of such ``singularities" at the origin; see
Section\;\ref{S3}.

Recall again that, also, in representing our pencil approach (to
be used later on for other linear and nonlinear elliptic
problems), in dealing with the classic Laplace equation, we will
re-discover several standard and well-known facts from any
text-book on linear operators. However, what is more important,
the general structure of such an approach will proceed in what
follows.

\subsection{Blow-up scaling and rescaled equation}

First we show the required transformations with which we will obtain the pencil operators 
that eventually will provide us with the behaviour of the solutions at the tip of the crack for the Laplace problem \eqref{Lap1}.

Thus, assuming the crack configuration as in Figure \ref{figcrack}, we
introduce the following rescaled variables, corresponding to
``blow-up" scaling near the origin $0$:
  \be
  \label{resc1}
  \mbox{$
u(x,y)= w(z ,\t), \quad \mbox{with} \quad z=  x/{(-y)}
 \quad \mbox{and} \quad  \t = - \ln
(-y) \forA y<0,
 $}
  \ee
   to get the rescaled operator 
 \be
 \label{www.20}
\D_{(x,y)}u= \eee^{2 \t} \,\big[ D_\t^2+ D_\t+ 2 z D^2_{z \t} +
(1+z^2)D_z^2 +2z D_z\big]w \equiv \D_{(z ,\t)}w.
 \ee
Therefore, we arrive at the
equation
 \be
 \label{www.2}
  w_{\t\t} + w_\t + 2z w_{z  \t} = {\bf A}^* w \equiv -(1+z
^2) w_{zz }- 2 z w_z .
 \ee
 
\begin{remark} Note that $\AAA^*$ is symmetric in the standard (dual) metric of
$L^2(\re)$,
 \be
 \label{aa1}
 \AAA^* \equiv - D_z[(1+z^2)D_z],
  \ee
  though we are not going to use this. Indeed, for our crack
  purposes, we do not need eigenfunctions of the ``adjoint"
pencil, 
since we are not going to use
eigenfunction expansions of solutions of the PDE \ef{www.2}, where
bi-orthogonal basis could naturally be wanted.

\end{remark}

This blow-up analysis of \ef{www.2}
assumes a kind of  ``elliptic evolution" approach for elliptic
problems, which  is not well-posed in the Hadamard's sense, but,
in fact, can trace out the behaviour  of necessary global orbits
that reach and eventually decay to the singularity point
$(z,\t)=(0,+\iy)$. 

By the crack condition \ef{Dir1}, we
look for vanishing solutions: in the mean and  uniformly on
compact subsets in $z$,
 \be
  \label{zer1}
   w(z,\tau) \to 0 \asA \t \to +\iy .
    \ee
Therefore, under the rescaling \eqref{resc1}, we have converted
the singularity point at $0$ into an asymptotic convergence
when $\tau \to \infty$.

Hence,  we are forced to describe a very thin family of
solutions for which
we will describe their possible nodal
sets to settle the multiple crack condition in \ef{Lap1}. This
corresponds to Kondratiev's ``evolution" approach \cite{Kon1,
Kon2} of 1966, though it was there directed to different {\em
boundary point regularity} (and asymptotic expansions) questions,
while the current crack problem assumes studying the behaviour at
an {\em internal} point $0 \in \O$ such as the tip of the multiple
crack under consideration. We will show first that this {\em
internal crack problem} requires polynomial eigenfunctions of
different pencils of linear operators, which were not under
scrutiny in Kondratiev's works.
Indeed, in doing so, we will ``re-discover" classic harmonic
polynomials, which will play a key role. Therefore, in studying
such classical objects, we could omit many technical details, but,
anyway, prefer to keep some of them for the sake of comparison and general
logic.

\subsection{Quadratic pencil and its polynomial eigenfunctions}

Now, we ascertain the quadratic pencil operator associated with equation \eqref{www.2}. So that, 
 looking, as usual in linear PDE theory,  for solutions of
\ef{www.2} in separate variables
 \be
\label{YYY.112S}
 w(z , \t)= {\mathrm e}^ {\l \t} \psi^*(z ) \whereA {\rm Re} \, \l
 <0 \,\,\, \mbox{by (\ref{zer1})},
 \ee
yields the eigenvalue problem for a {\em quadratic pencil}
of non self-adjoint operators, 
 \be
 \label{Pen.1}
 {\bf B}_\l^* \psi^* \equiv  \{\l(\l+1)I+ 2 \l z  D_z  - {\bf A}^*\} \psi^*=0
 \,\, \mbox{or}\,\, (1+z^2)(\psi^*)''+2(\l+1)z (\psi^*)' +\l(\l+1)
 \psi^*=0.
   \ee
   
    \vspace{0.2cm}

\begin{remark}
\label{re34}
The  second-order operator ${\bf A}^*$ is singular  at the
infinite points $z= \pm \infty$, so this is a singular quadratic
pencil eigenvalue problem. Since the linear first-order operator
in \ef{Pen.1}, $z D_z$, is not symmetric
 in $L^2$, we are not obliged to attach the whole operator to any
 particular functional space.
 Therefore, the behaviour as $z \to \infty$  is not that
crucial, and any $L^2_\rho$-space setting  with $\rho(z) \sim
{\mathrm e}^{- a z^2}$ (or ${\mathrm e}^{-a |z|}$), $a>0$ small,
would be enough.

Indeed, if the solution of the problem \eqref{Lap1} is smooth in
certain weighted spaces $H^1_\rho$ or $L^2_\rho$, we claim that,
then, the eigenfunctions $\psi^*$ of the operator \eqref{Pen.1}
are also analytic at infinity.
\end{remark}

\begin{remark} The differential part in \ef{Pen.1}
can be reduced to a symmetric form in a weighted
$L^2_{\rho_\l}$-metric:
  \be
  \label{symm12}
  \tex{
  (1+z^2) D_z^2+ 2(\l+1) z D_z \equiv (1+z^2) \frac
  1{\rho_\l} D_z(\rho_\l D_z) \whereA \rho_\l =(1+z^2)^{\l+1}.
   }
   \ee
 Note that this weighted metric has an essential dependence
 on the  {\em a priori} unknown eigenvalues.

\end{remark}

Furthermore, we cannot forget that once the rescaling \eqref{resc1} is performed, these eigenfunctions of the quadratic pencil operator $(\ref{Pen.1})$ 
are actually harmonic polynomials (just introducing the variables \eqref{resc1}). Then, as usual in orthogonal polynomial theory, we can now state the
following property for the eigenfunctions of the adjoint pencil \eqref{Pen.1} with respect to its family of eigenvalues that will be determined below;
see \cite{Det,Fun} for details about this Sturm--Liouville Theory as well as Subsection $3.6$ below.

\begin{proposition}
 \label{Pr.PolP}
 The only acceptable eigenfunctions of the adjoint pencil $(\ref{Pen.1})$
are finite polynomials.
 \end{proposition}

 \vspace{0.2cm}
 
 Although, it is well known that these eigenfunctions of the quadratic pencil operator $(\ref{Pen.1})$ 
are finite polynomials, it should be pointed out that this is associated with the interior elliptic
regularity.

Indeed, the blow-up approach under the rescaling \eqref{resc1} just
specifies local structure of multiple zeros of analytic functions
at 0, and since all of them are finite (we are assuming \eqref{str1} with a finite number of cracks) we must have finite
polynomials only. 

Of course, there are other formal eigenfunctions
(we will present an example; see \eqref{001} below), but those, in the limit as $\t \to +
\iy$ in \ef{YYY.112S}, lead to non-analytic (or even discontinuous)
solutions $u(x,y)$ at 0, that are non-existent.

 On the other hand, since our pencil approach, currently, is
nothing more than re-writing via scaling the standard
Sturm--Liouville eigenvalue problem for harmonic polynomials (see Subsection\;\ref{stlio}), it is
quite natural  to deal with nothing other than them, which, thus,
should be re-built in terms of the scaling variable $z$.

 Moreover, the next lemma shows 
 the corresponding point spectrum of the pencil \eqref{Pen.1}.

\begin{lemma}
 The quadratic pencil operator \eqref{Pen.1} admits two families of eigenfunctions 
 $$
  \tex{\psi_{l_+}^*(z)\equiv \psi_{l,1}^*(z) \quad \hbox{and}
  \quad \psi_{l_-}^*(z)\equiv \psi_{l-1,2}^*(z), \quad \hbox{for any}\quad  l=m,m+1,\cdots}
  $$
associated with two corresponding families of eigenvalues 
 \be
 \label{ei11}
 \tex{
  \l_l^+=-l, \quad l=1,2,3,... \andA \l_l^-=-l-1, \quad
  l=0,1,2,3,... \,,
  }
  \ee
 \end{lemma}

\begin{proof}
 In order to find the corresponding point spectrum of the pencil we look for $l$th-order polynomial eigenfunctions of the form
 \be
 \label{eig00}
  \tex{
\psi_l^*(z)=z^l+a_{l-2}z^{l-2}+ a_{l-4}z^{l-4}+... =
\sum\limits_{k=l,l-2,...,0} a_k z^k, \quad (a_l=1),
  }
   \ee
   that we already know they are harmonic polynomials. Substituting \eqref{eig00} into (\ref{Pen.1})
   and evaluating the higher order terms
 yields the following quadratic equation for eigenvalues:
  \be
  \label{eig11}
 O(z^l): \quad  \l_l^2 + (2l+1) \l_l + l(l + 1)=0.
   \ee
Solving this characteristic equation yields the two families of real
negative eigenvalues under the expression \eqref{ei11}
  associated with two families of eigenfunctions denoted by
 \be
 \label{harpol}
  \tex{\psi_{l_+}^*(z)\equiv \psi_{l,1}^*(z) \quad \hbox{and}
  \quad \psi_{l_-}^*(z)\equiv \psi_{l-1,2}^*(z), \quad \hbox{for any}\quad  l=m,m+1,\ldots,}
  \ee
for convenience.  
\end{proof}

 \vspace{0.2cm}
 
 The next result calculates those  (re-structured harmonic)
  polynomials \eqref{eig00}, \eqref{harpol} as the corresponding eigenfunctions of the pencil. 

 \begin{theorem}
  \label{Th.1}
 The quadratic pencil \ef{Pen.1} has two (admissible) discrete spectra
 $\ef{ei11}$ of real negative eigenvalues with the finite polynomial
 eigenfunctions given by \ef{eig00}, where the expansion
 coefficients satisfy a
 finite Kummer-type recursion corresponding to the operator in
$(\ref{Pen.1})$:
  \be
  \label{co1}
   \left\{
    \begin{array}{ccc} \tex{
    a_{k+2}= -
    \frac{k(k-1) +2(\l_l^\pm+1)k +\l_l^\pm(\l_l^\pm+1)
   }{(k+2)(k+1)} \,  a_{k},} & \hbox{for any} & k=l,l-2,...,2,
         \\
        \tex{ a_1= -\frac{6}{ 2(\l_l^\pm+1)  +\l_l^\pm(\l_l^\pm+1)}a_3} & \hbox{and} & \tex{a_0 =- \frac{2}{\l_l^\pm(\l_l^\pm+1)} a_2.}
        \end{array}\right.
    \ee
  \end{theorem}

\begin{proof}
    It is clear by \eqref{ei11} that the quadratic pencil \ef{Pen.1} has two discrete spectra
    of real negative eigenvalues with two families of finite polynomial
 eigenfunctions\footnote{Note that, within this pencil ideology,
 the eigenfunctions are ordered in an unusual manner, unlike
 the standard harmonic polynomials.}
$$
\{\psi_{l_+}^*(z)\},\quad \{\psi_{l_-}^*(z)\},\quad \hbox{such that}\quad \psi_{l_+}^*(z)
\equiv \psi_{l,1}^*(z)\quad \hbox{and}\quad \psi_{l_-}^*(z)\equiv \psi_{l-1,2}^*(z),
$$
given by \ef{eig00} and corresponding associated with the two families of eigenvalues $\l_l^+$ and $\l_l^{-}$ .
Substituting $\psi_l^*=\sum\limits_{k\geq 0}^l a_k z^k$, for any $l\geq 0$, into (\ref{Pen.1})
    we find that, for any $\l$,
$$
 \tex{
 (1+z^2) \sum\limits_{k\geq 2}^l k(k-1) a_k z^{k-2} + 2(\l+1) \sum\limits_{k\geq 1}^l k a_k
z^{k}+\l(\l+1) \sum\limits_{k\geq 0}^l a_k z^{k}=0,
 }
 $$
and hence,
\be
\label{coefla}
\begin{split}
\sum\limits_{k\geq 2}^l & \left[(k+2)(k+1) a_{k+2}+ k(k-1) a_k+2(\l+1)ka_k +\l(\l+1)a_k\right] z^k \\ & +
\left[ 6a_3 + [2(\l+1)  +\l(\l+1)]a_1\right] z+ 2a_2+\l(\l+1) a_0=0.
\end{split}
\ee
Therefore, evaluating the coefficients we find that
$$\left\{\begin{array}{l}
    (k+2)(k+1) a_{k+2}+ k(k-1) a_k+2(\l+1)ka_k +\l(\l+1)a_k=0,\quad k=l,l-2,...,2,\\
    6a_3 + [2(\l+1)  +\l(\l+1)]a_1=0,\\
    2a_2+\l(\l+1) a_0=0,\end{array}\right.
    $$
    and we arrive at  \eqref{co1}, completing the proof. 
    \end{proof}

\begin{remark}
    \rm{Alternatively, we also have that
    \[\tex{
     a_{l-2n}= -
    \frac{(l-2n+2)(l-2n+1)
   }{(l-2n)(l-2n-1)+2(2\l_l^\pm+1)(l-2n)+\l_l^\pm(\l_l^\pm+1)} \,  a_{l-2n+2},
    \quad n=1,2,...,[\frac l2]; \quad a_l=1.
    }
    \]
    }
\end{remark}

Note that even when
discrete spectra coincide excluding the first eigenvalue
$\l_l^-$, and, more precisely,
\[
\l_l^-= \l_l^+-1= \l_{l+1}^+ \quad l=1,2,3,...\,,
\]
we still have two different families of eigenfunctions. For future
convenience and applications for the crack problem for $m=1,2,3$,
and 4, (with $m=l$), we present the first four
eigenvalue-eigenfunction pairs of both families of eigenfunctions
for the pencil \ef{Pen.1}, which now are ordered with respect to
$\l=-l$, $l=0,1,2,...$:
\be
\label{pol111}
 \begin{split}
\l_0=0, &  \quad \mbox{with} \,\,\, \psi^*_0(z)=1 \,\,\,(\ne 0);\\
\l_1=-1, & \quad \mbox{with} \,\,\, \psi^*_{1,1}(z)=z, \quad
\psi^*_{0,2}(z)=1 \,\,\,(\ne 0);\\
 \l_2=-2, & \quad \mbox{with} \,\,\,\psi^*_{2,1}(z)=z^2-1,
 \,\,\, \psi^*_{1,2}(z)=z; \\
  \l_3=-3, & \quad \mbox{with} \,\,\,
 \psi^*_{3,1}(z)=z^3 - 3z,\,\,\,
\psi^*_{2,2}(z)=3z^2-1; \\
 \l_4=-4, & \quad \mbox{with} \,\,\,
\psi^*_{4,1}(z)=z^4 - 6z^2+1,\,\,\,
 \psi_{3,2}^*(z)=z^3 -z;\,\,\, \mbox{etc.}
 \end{split}
 \ee

 \vspace{0.2cm}

\noi{\bf Remark: about transversality.}
These (harmonic) polynomials satisfy the Sturmian property (important for applications) in the sense that
 each polynomial $\psi_{m}^*(z)$ has precisely $m$ {\em transversal}
 zeros. 
 For Hermite polynomials, this result
  was proved by Sturm already in 1836
  \cite{St}; see further historical comments in
  \cite[Ch.~1]{GalGeom}.

 \vspace{0.2cm}

\noi{\bf Remark:
about analyticity.} Obviously, we exclude,
in the first line of \ef{pol111}, the first eigenfunction
$\psi^*_0(z) \equiv \psi^*_{0,1}(z) \equiv 1$, since it does not
vanish and has nothing to do with a multiple zero formation.
However, for $\l=0$ in \ef{Pen.1}, there exists another obvious
bounded analytic solution having a single zero:
 \be
 \label{001}
 (1+z^2) (\psi^*)''+ 2z (\psi^*)'=0 \LongA \tilde\psi^*(z)= \tan^{-1}z \to \pm \pi/2 \asA z \to \pm \iy.
 \ee
This $\tilde \psi^*(z)$ belongs to any suitable $L^2_\rho$-space
(of polynomials). However, it becomes irrelevant due to another
regularity reason: passing to the limit in the corresponding
expansion of $u(x,y) \equiv w(y,\t)$ \ef{YYY.112S} as $\t \to
+\iy$ ($y \to -0$) yields the discontinuous limit ${\rm sign}\,x$,
i.e. an impossible trace at $y=0$ of any analytic solutions of
the Laplace equation.

\subsection{Nonexistence result for crack problem
}

Next we ascertain how the family of admissible cracks should lead to the existence of solutions for the crack problem \eqref{Lap1}.

We have that sufficiently ``ordinary" polynomials are
always complete in any reasonable weighted $L^2$ space, to say
nothing about the harmonic ones; see \cite[p.~431]{KolF}.
Moreover, since our  polynomials are not that different from
standard harmonic (or Hermite) ones, this implies the completeness
in such spaces. So that, sufficiently regular solutions of
\ef{www.2}
should admit the corresponding
eigenfunction expansions over the polynomial family pair 
$$\Phi^*=
\{\psi_{l,1}^*,\psi_{l-1,2}^*\},$$ in the following sense. Bearing in mind two discrete
spectra
\ef{ei11}, the general expansion
has the form
 \be
 \label{exp11}
  \tex{
  w(z,\t)= \sum_{(k \ge l)} \eee^{-k \t} [c_k\psi_{k,1}^*(z)+ d_k \psi_{k-1,2}^*(z)],
  }
  \ee
 where two collections of expansion coefficients $\{c_k\}$ and
 $\{d_k\}$,
  depending on boundary
 data on $\O$, are presented. 
 
 We did not need to develop an
 ``orthonormal theory" of our polynomials, which should specify the
 expansion coefficients in \ef{exp11}, for a given solution
 $u(x,y)$ (though specifying all the coefficients declare the
 whole family of $u$ with such cracks at 0). Indeed, dealing with
 orthonormal harmonic polynomials, we just have a standard
 expansion for harmonic functions, and obtain \ef{exp11} by introducing the scaling
 blow-up variables \ef{resc1}.

 Note that as mentioned in the introduction the linear combination \eqref{exp11} arises naturally
from the spectral theory of the operator (in this case the Laplacian, later on the bi-Laplacian).
Indeed, for the Laplacian $u$ is harmonic in $B_1\setminus \G$ and can be decomposed
by homogeneous harmonic functions, here denoted by $\psi_{k,1}^*$ and $\psi_{k-1,2}^*$. Even facing
a difficult regularity problem in $\O\setminus \G$ (at the singularity boundary point) we are in the context analysed in
\cite{AL}, so that
 \[
 \tex{w(z,\t)=  \eee^{-k \t} \psi^*(z),}
 \]
 for an orthonormal basis $\{\psi_{k,1}^*,\psi_{k-1,2}^*\}$ of Hermite-type polynomials eigenfunctions. Hence, we find that
 our solutions are decompositions of the form \eqref{exp11}.

 Moreover, in view of sufficient regularity of ``elliptic
orbits" (via standard interior elliptic regularity), such
expansion is to converge not only in the mean (in $L^2_\rho$,
with an exponentially decaying weight at infinity), but also
uniformly on compact subsets. This allows us now to prove our
result on nonexistence for the crack problem.

\begin{theorem}
\label{Th.2} Let the cracks $\G_1$,...,$\G_m$ in \ef{curv1} be
asymptotically given by $m$ different straight lines \ef{str1}.
Then, the following hold:

{\rm (i)} If all
$\{\a_k\}$ do not coincide with  all $m$ subsequent zeros of any
non-trivial linear combination
 \be
 \label{dmex}
c_l \psi^*_{l,1}(z)+d_l \psi^*_{l-1,2}(z), \quad \mbox{with} \quad
c_l^2+d_l^2 \ne 0 \whereA z= x/(-y),
 \ee
 of two families of (re-written harmonic)
polynomials $\psi_{l_+}^*(z)\equiv  \psi^*_{l,1}(z)$ and $\psi_{l_-}^*(z)\equiv  \psi^*_{l-1,2}(z)$ defined by
\ef{eig00}, \ef{co1} for any $l=m,m+1,...$ and arbitrary constants
$c_l, \, d_l \in \re$, then the multiple crack problem \ef{Lap1}
cannot have a solution for any boundary Dirichlet data $f$ on
$\O$.

{\rm (ii)} If, for some $l$, the distribution of zeros in {\rm
(i)} holds and a solution $u(x,y)$ exists,
 then
 \be
 \label{as11ex}
  \tex{
  |u(x,y)| = O(|x,y|^{l})    \asA (x,y) \to (0,0).
  }
  \ee

 \end{theorem}

 \begin{proof} Condition \ef{str1} implies that the elliptic
 ``evolution" problem while approaching the origin actually occurs on
 compact, arbitrarily large subsets for $ x/(-y) \equiv z$.
 Since we have converted the singularity point at $(0,0)$ into an
 asymptotic point when $\tau \to \infty$.

 Therefore, \ef{exp11} gives all possible types of such a decay.
 Hence, choosing the first non-zero expansion coefficients $c_l$, $d_l$ in
 \ef{exp11}, that satisfies $c_l^2+d_l^2 \ne 0$, we obtain a sharp asymptotic behaviour of this solution
  \be
  \label{st.2}
  w_l(y,\t)=  \eee^{-l \t}[c_l \psi_{l,1}^*(z)+d_{l} \psi_{l-1,2}^*] + O(\eee^{-(l+1) \t})
  \asA \t \to +\iy.
  \ee
  Obviously,  
then the straight-line cracks \ef{str1} correspond to zeros of
the linear combination 
$$\tex{c_l \psi_{l,1}^*(z)+d_{l}
\psi_{l-1,2}^*(z)},$$ 
and the full result is straightforward since by the blow-up scaling if all the $\a_k$ do not 
coincide with zeros of the previous linear combination \eqref{dmex} (harmonic polynomials) the crack problem 
does not have a solution, since
$$ z=\frac{x}{-y}=\a_k (1+o(1)), \,\, y \to 0.$$
 Otherwise, if there is some $l$ for which all the $\a_k$
coincide with zeros of \eqref{dmex} we find that the crack problem \eqref{Lap1} possesses a solution and \eqref{as11ex} is satisfied.
  The proof is complete. \quad
\end{proof}

\begin{remark}
Remember that thanks to the rescaling \eqref{resc1}, we have converted
the singularity point at $0$ into an asymptotic convergence
when $\tau \to \infty$.
\end{remark}

\begin{remark} 
Of course, one can ``improve" such nonexistence results. For
instance, if cracks have an asymptotically small ``violation" of
their straight line forms near the origin, which do not correspond to
the exponential perturbation in \ef{st.2} (if $c_{l+1}$ and
$d_{l+1}$ do not vanish simultaneously; otherwise take the next
non-zero term), then the crack problem is non-solvable.
\end{remark}

Overall, we can state the following most general conclusion. 

\begin{corollary}
For almost every straight-line crack \ef{curv1}, the crack problem
\ef{Lap1} cannot have a solution for any Dirichlet data $f$,
provided that the crack behaviour at the origin is not consistent
with all the eigenfunction expansions \ef{exp11} via the above
(harmonic) polynomials.
\end{corollary}

\ssk

Finally, concerning the admissible boundary data for such
$l$-cracks at the origin, these are described by all the
expansions \ef{exp11} with arbitrary expansion coefficients
excluding the first ones $c_l$, $d_l$, which are fixed by the
multiple crack configuration (up to a common non-zero multiplier)
and satisfying $c_l^2+d_l^2 \ne 0$.


 \subsection{Extensions to semi-linear equations: a regular perturbation}

With the idea in mind of extending the techniques performed above to non-linear problems
we show a couple of examples. Especially interesting is the application of pencil operators for non-linear equations 
since in most cases this creates problems. See for example \cite{CGpLap}. 

As a key explaining example, consider the semi-linear Laplace
equation
 \be
 \label{sem1}
 \D u + |u|^{p-1}u=0  \whereA p>1.
 \ee
One can see that performing the same rescaling \ef{resc1} on \ef{sem1}, in
view of \ef{www.20}, will lead to the following exponentially
small perturbation of the rescaled Laplacian one: as $\t \to
+\iy$,
 \be
 \label{sem2}
 \big[ D_\t^2+ D_\t+ 2 z D^2_{z \t} +
(1+z^2)D_z^2 +2z D_z\big]w  + \eee^{-2 \t} |w|^{p-1}w=0.
 \ee
Obviously, then, on any leading asymptotic pattern given by stable
subspaces in \ef{exp11}, the last nonlinear term in \ef{sem2} is
negligible, so cannot affect the types of decay patterns at the
origin.

\subsection{Extensions to semi-linear equations: a singular perturbation}

It is seen from the previous example that in order to involve the
nonlinear term in a formation of multiple zeros at the origin, it
must be singular nearby, which happens for this model:
 \be
 \label{sem3}
  \tex{
 \D u + \frac{|u|^{p-1}u}{x^2+y^2}=0 \quad (p>1).
 }
  \ee
  Then by the same rescaling, instead of \ef{sem2},
  one obtains the following operator
 \be
 \label{sem4}
  \tex{
 \big[ D_\t^2+ D_\t+ 2 z D^2_{z \t} +
(1+z^2)D_z^2 +2z D_z\big]w  + \frac{ |w|^{p-1}w}{1+z^2}=0,
 }
 \ee
so that the non-linear term does not have an exponentially decaying
multiplier such as in \ef{sem2}.

\ssk

\noi{\sc Stationary profiles.} Firstly, it is straightforward to
consider bounded {\em stationary} solutions of \ef{sem4}:
 \be
 \label{sem41}
  \tex{
 w(z,\t)=f(z) \LongA
(1+z^2) f'' +2z f'  + \frac{ |f|^{p-1}f}{1+z^2}=0 \inB \re.
 }
 \ee
 In order to pose necessary conditions at $z=\iy$, consider the
 operator linearised at $f=0$, $z=\iy$, that yields the following
 roots of the characteristic equation:
  \be
  \label{sem411}
  \tex{
  (1+z^2) f'' + 2 z f'=0, \quad f=z^m \LongA m^2+m=0 \LongA m_1=-1,
  \,\,\, m_2=0,
  }
  \ee
  evaluating again the higher order terms. Therefore, we first consider \ef{sem41} with the following conditions
 as $z \to \iy$,
  \be
  \label{sem42}
   \tex{
  f(z)=O(\frac 1z);
   \quad f'(0)=0 \,\,\,(\mbox{symmetry}) \,\,\, \mbox{or}\,\,\, f(0)=0
   \,\,\,(\mbox{anti-symmetry}).
   }
   \ee
 Thus, with $m_1=-1$ the last
condition in \ef{sem41} corresponds to ``dipole-like" profiles.
Both symmetry and anti-symmetry conditions are associated with the
fact that  the ODE \ef{sem41} is invariant under the reflection
 \[
 z
\mapsto -z, \quad f \mapsto -f,
 \]
 which allows us to extend solutions for $z>0$ to $\{z<0\}$ in
 symmetric or anti-symmetric ways.
Of course, stationary nonlinear eigenfunctions \ef{sem41}
correspond to usual straight-line nodal sets.

A symmetric stationary profile $f(z)$ satisfying \ef{sem41} is
shown in Figure \ref{Fsymm1} for the cubic case $p=3$. In Figure
\ref{Fdip1}, we show a dipole-like profile as a solution of the
ODE in \ef{sem41}, again, for the cubic nonlinearity with for
$p=3$.


 \begin{figure}
\centering
\includegraphics[scale=0.75]{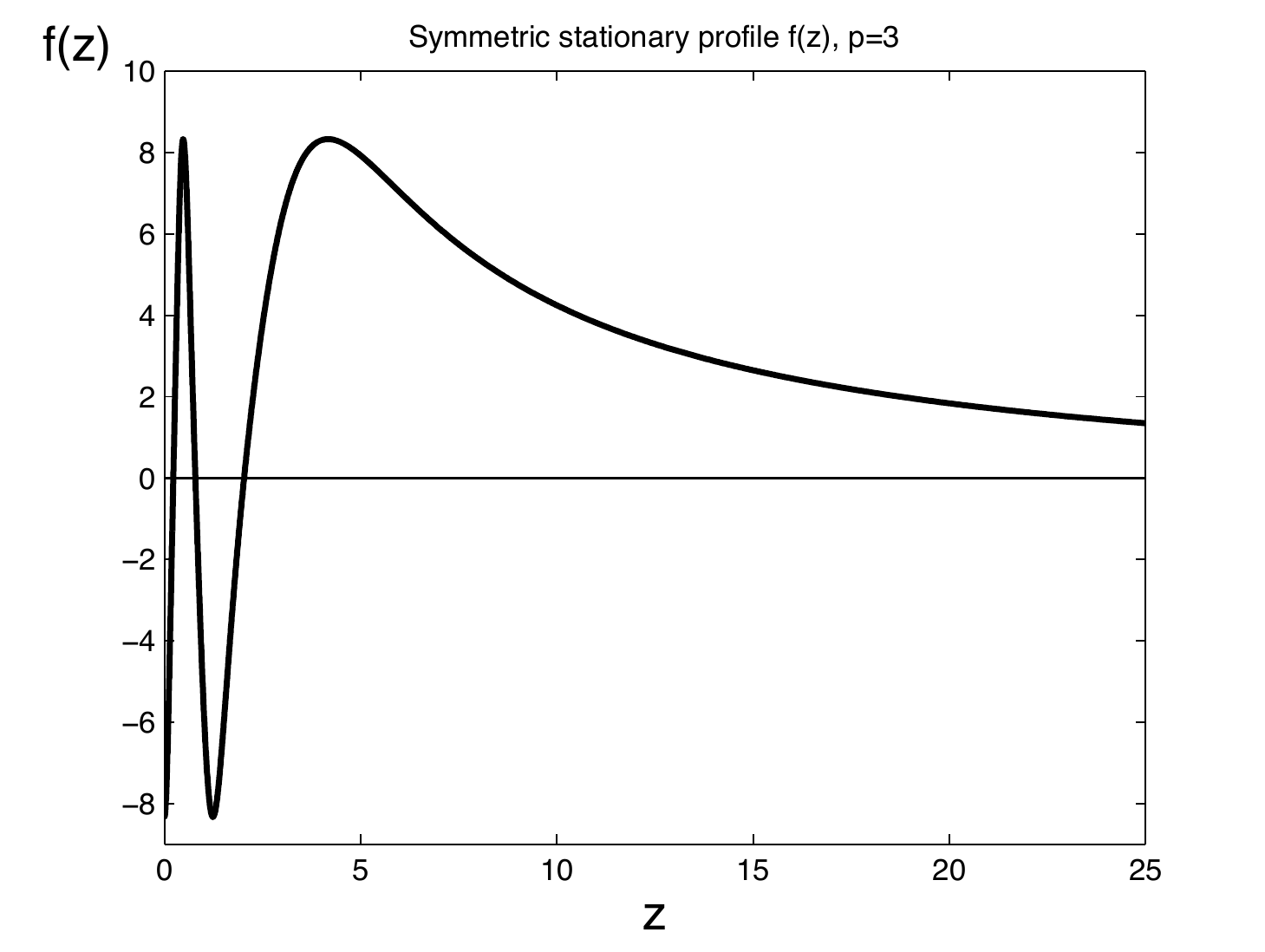}  
\vskip -.3cm
  \caption{An example of a symmetric bounded stationary  self-similar solution $f(z)$ of \ef{sem41} for $p=3$.}
 \label{Fsymm1}
\end{figure}



 \begin{figure}
\centering
\includegraphics[scale=0.75]{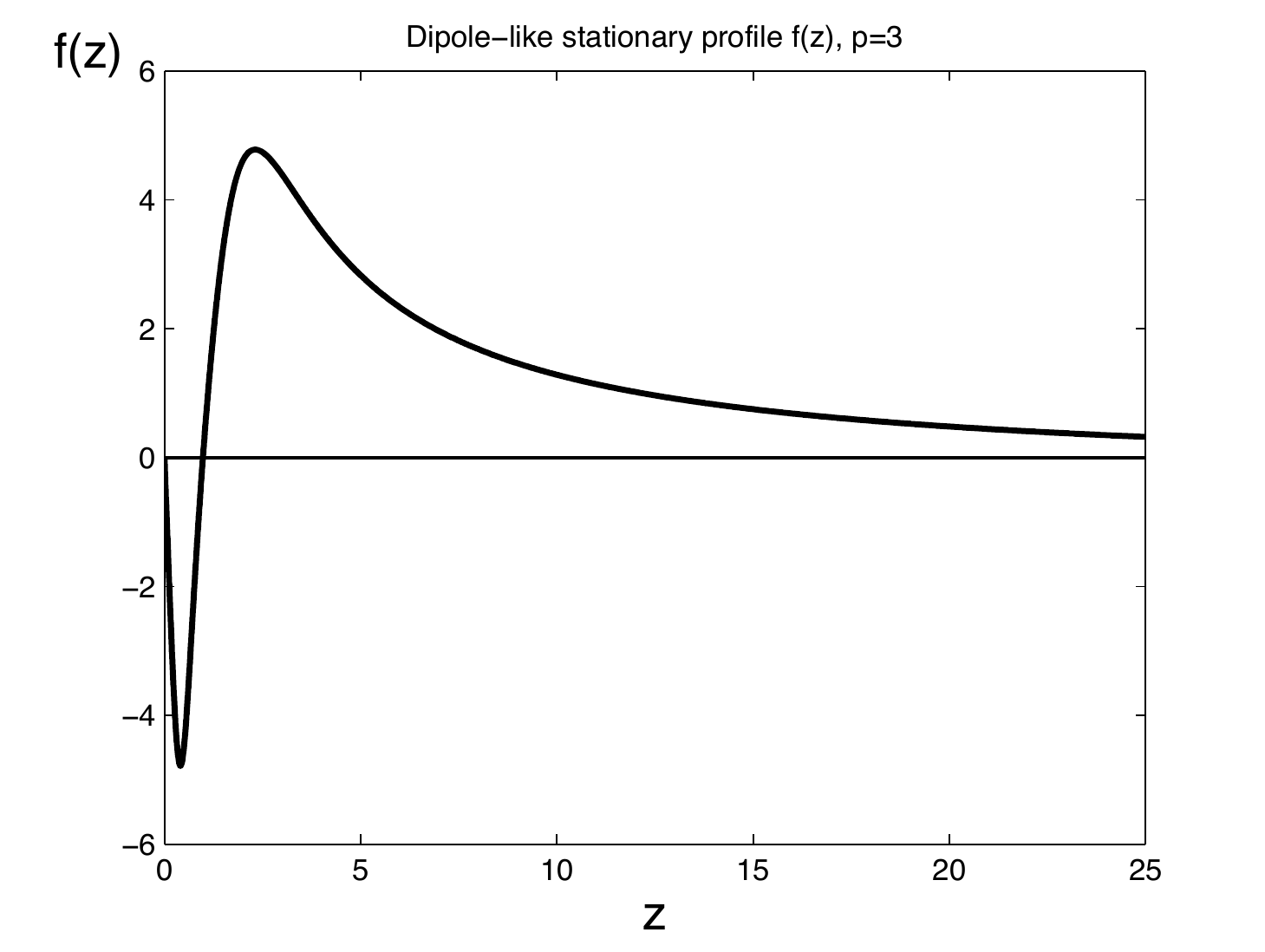}  
\vskip -.3cm
  \caption{An example of a dipole-like  stationary  self-similar solution $f(z)$ of \ef{sem41} for $p=3$.}
 \label{Fdip1}
\end{figure}


Also, the second root $m_2=0$ in \ef{sem411}
allows us to consider stationary profiles satisfying
 \be
 \label{sem43}
 f(+\iy)=1.
  \ee
  Figure \ref{Fsd1} shows that such profiles exist for $p=3$, for
  both symmetric and dipole-like (the dash-line) cases. Overall,
  those examples exhibit a vast variety of nonlinear
  eigenfunctions with different nodal sets for elliptic equations
  with singular nonlinear perturbations.


 \begin{figure}
\centering
\includegraphics[scale=0.75]{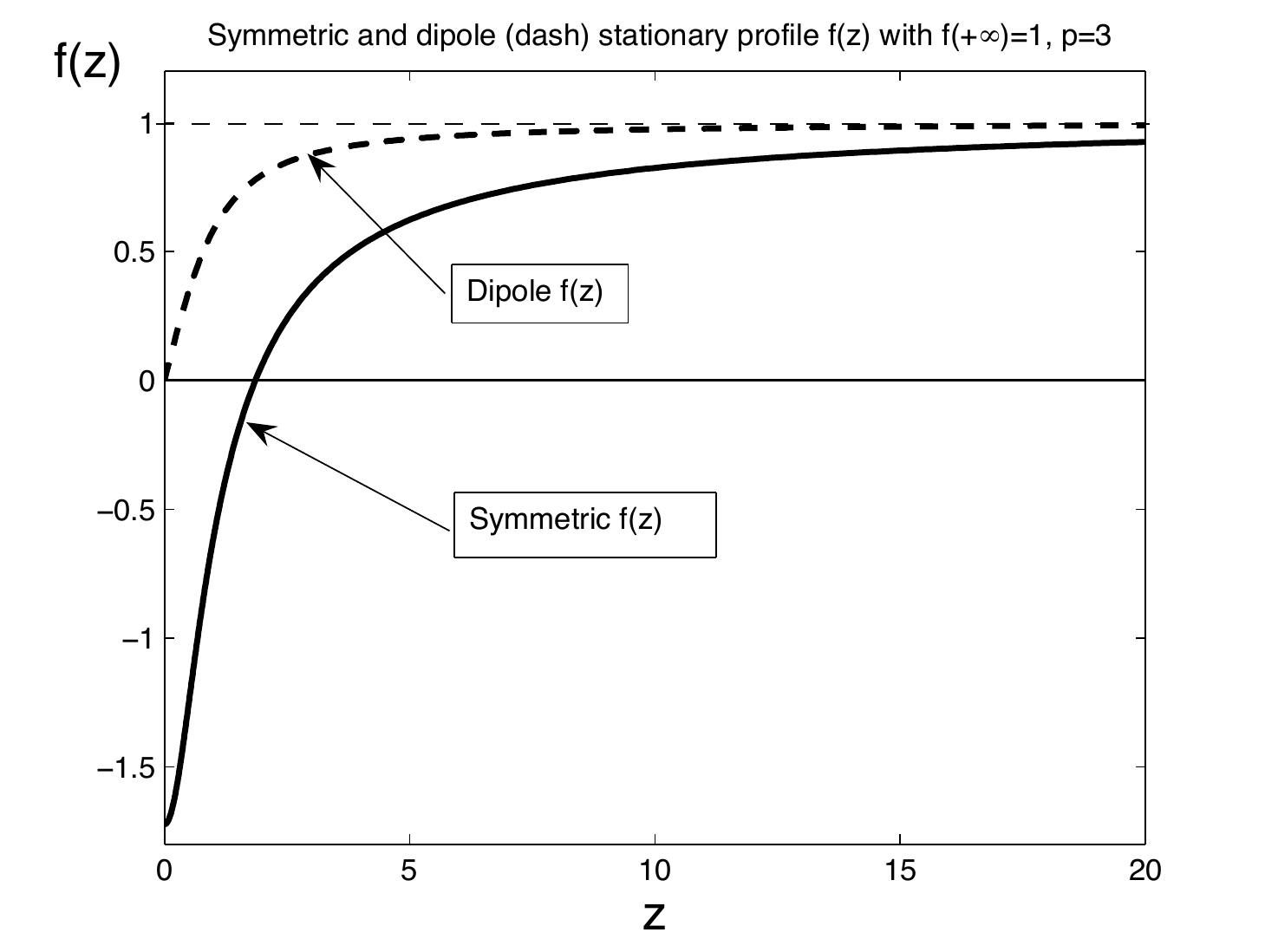}  
\vskip -.3cm
  \caption{Example of symmetric and dipole-like  stationary  self-similar solution $f(z)$ of \ef{sem41}, \ef{sem43} for $p=3$.}
 \label{Fsd1}
\end{figure}


\ssk

\noi{\sc Quasi-stationary self-similar solutions.}
 Secondly, as other  ``nonlinear
eigenfunctions" depending on $\t$, we can look for an {\em
approximate} self-similar solution of a standard form:
 \be
 \label{sem5}
  \tex{
 w(z,\t)= \t^\a f(\xi) \whereA \xi= \frac{z}{\t^\b} \whereA \b=
 \frac {\a(p-1)}2,
 }
 \ee
 and $\a>0$ is an arbitrary fixed exponent. It is clear that the
 evolution structure \ef{sem5} is {\em quasi-stationary}, since
 all three first time-dependent derivatives, after scaling, are negligible as $\t
 \to +\iy$, of the order, at least, $\sim O(\frac 1 \t)$, in comparison with the three other stationary ones. Then, the
 ODE for $f$ asymptotically takes the form (cf. that in \ef{sem41})
  \be
  \label{sem6}
   \tex{
     \xi^2 f''+ 2 \xi f' + \frac{|f|^{p-1}f}{\xi^2}=0,
    }
     \ee
(note that $\a$ does not affect this equation). One can see that
\ef{sem6} admits solutions with the same decay at infinity:
 \be
 \label{sem7}
  \tex{
   f(\xi)= O(\frac 1 \xi)\to 0 \asA \xi \to \iy.
   }
   \ee
   At $\xi=0$, the operator is singular, so one cannot put any
   definite condition on it, and we just require  $f$ to be
   bounded. Again, we are not going to study this ODE problem in
   any detail. In Figure \ref{Fx21} we just present such a
   self-similar profile for $p=3$. Note that it is oscillatory as
   $\xi \to 0$, so that the nodal set of such an unbounded
   ($\a>0$) pattern consists of an infinite number of zero
   curves, with the following non-standard behaviour near the origin
   (cf. \ef{str1}: here, there is a $\log$-type perturbation of the crack geometry
   for such nonlinear patterns):
    \be
    \label{sem8}
     \tex{
     x_k= \xi_k \, (-y) |\ln(-y)|^\b \whereA k=1,2,3,..., \quad
     f(\xi_k)=0.
      }
      \ee
Indeed, this is a rather strange 
example of a
multiple crack (while the solution gets unbounded at 0), but one
should remember that, here, we are talking about a {\em strongly
singularly} perturbed Laplace operator \ref{sem3}, for which a
proper statement of the Dirichlet problem deserves certain
attention.


 \begin{figure}
\centering
\includegraphics[scale=0.75]{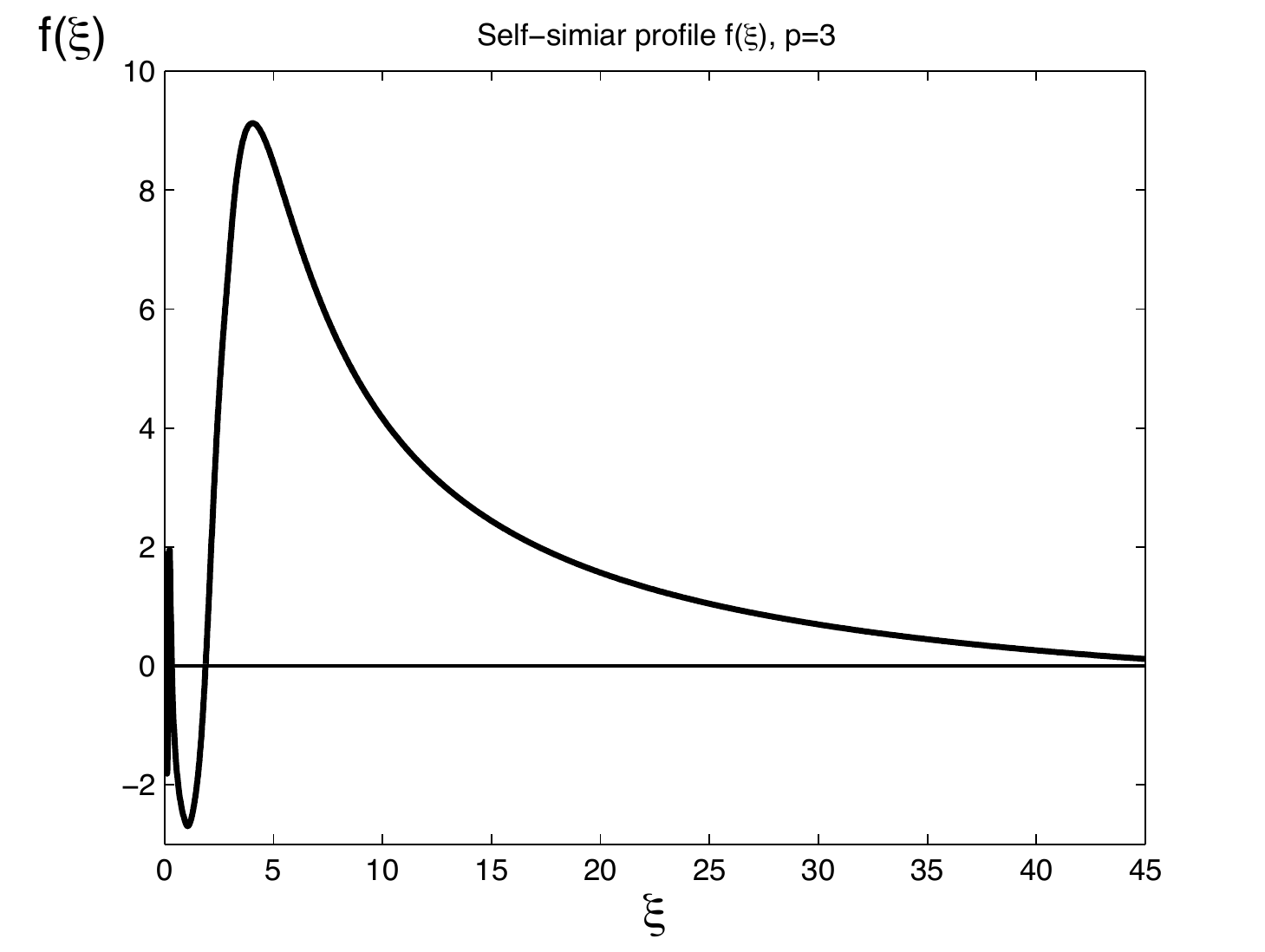}  
\vskip -.3cm
  \caption{An example of a bounded oscillatory self-similar solution $f(\xi)$ of \ef{sem6}, \ef{sem7} for $p=3$.}
 \label{Fx21}
\end{figure}



\subsection{A comment: a standard Sturm--Liouville form of the pencil}
\label{stlio}


Recall the classic fact: harmonic polynomials are eigenfunctions
of a standard Sturm--Liouville problem. Therefore, obviously, our
pencil eigenvalue problem must admit a reduction to a similar one.
It is easy to see that, e.g., this can be achieved by the
transformation
 \[
 \psi^*(z)=(1+z^2)^\g \varphi(z),
  \]
  with a
parameter $\g \in \re$ to be determined. Then  we find that the
operator \eqref{Pen.1} can be written as
\be
\label{stuli}
\begin{split}
 (1+z^2)^\g [ \l (\l+1)\varphi &
 +4(\l+1) \g z^2 (1+z^2)^{-1}\varphi +2(\l+1) z \varphi'+2\g \varphi  \\ &+
4z^2\g(\g-1)(1+z^2)^{-1} \varphi  +4z \g \varphi'
+(1+z^2)\varphi'']=0,
\end{split}
\ee
since
\[
(\psi^*)'(z)=2\g z (1+z^2)^{\g-1}\varphi(z)+(1+z^2)^{\g}\varphi'(z),
 \quad \mbox{and}
 \]
\[
(\psi^*)''(z)= 2\g  (1+z^2)^{\g-1}\varphi(z)+4\g (\g-1) z^2
(1+z^2)^{\g-2}\varphi(z)+ 4\g z (1+z^2)^{\g-1}\varphi'(z)+
(1+z^2)^{\g}\varphi''(z).
  \]
   To eliminate the necessary terms in
order to get a Sturm--Liouville problem, we have to cancel the
term containing $z \var'$, i.e. to require
   \[
 \tex{
  2
  \l+4\g+2\LongA
\g=-\frac{\l+1}{2}.
 }
 \]
  Now, rearranging terms for that specific $\g$ in the equation
\eqref{stuli}, so that the terms with $\varphi$ are given by
 \[
  \tex{\left(1-z^2(1+z^2)^{-1}\right)(\l+1)(\l-1)\varphi \equiv
  (1+z^2)^{-1} (\l+1)(\l-1),
  }
  \]
we arrive at a Sturm--Liouville problem of the form
\be
\label{StumLiv}
 \tex{
\mc{A} \varphi=\mu \varphi,\quad \hbox{where} \quad
\mc{A}=-(1+z^2)^2 \frac{{\mathrm d}^2}{{\mathrm d}z^2} \andA
\mu=(\l+1)(\l-1),
 }
\ee
  in the space of functions
  \[
 \tex{
  \mc{D}=L^2\big(  \re,\frac{{\mathrm d}z}{(1+z^2)^2}\big).
 }
  \]
  The
operator $\mc{A}$ is symmetric in a weighted $L^2$-space, so the
eigenvalues $\mu$ are real. Indeed, by classic Sturm--Liouville
theory, we also state that there exists an eigenfunction
associated with every eigenvalue $\mu_n$ such that
\[\mu_1<\mu_2<\cdots<\mu_n \to \infty.\] 
Associated with those
eigenvalues we have the eigenfunctions $\varphi_n$ which have exactly
$n-1$ zeros in $\re$ and are the so-called $n$-th fundamental
solution of the Sturm--Liouville problem \eqref{StumLiv} and form
an orthogonal basis in a specific weighted $L^2$-space, denoted by
$L^2_\rho$ for an appropriate weight (in this case
$\rho=(1+z^2)^{-2}$). Indeed, by classical spectral theory we can be
assured that the first eigenvalue $\mu_1$ is positive and, hence
all the others. Also, since the weight $\rho$ is integrable, i.e.
\[\tex{\int_{\re} \frac{{\mathrm d}z}{(1+z^2)^2} <\infty,}\]
by classical spectral theory we can confirm that the spectrum is formed by a 
discrete family of eigenvalues. 
Thus,
 our pencil eigenvalues are associated with standard $\mu$'s via
 the quadratic algebraic equation
 \[
 \mu=(\l+1)(\l-1),
 \]
 and the correspondence of eigenfunctions
 is straightforward. We do not need any further discussion, since, inevitably,
  once more, we are starting to re-discover classic textbook's facts on harmonic
 polynomials.

\section{Bi-Laplace equation and new types of admissible cracks}
 \label{S9}

According to \ef{www.20}, for the bi-Laplace problem \ef{bila}, we
need to solve the iterated rescaled Laplacian:
 \be
 \label{9.1}
  \D_{(z,\t)} \D_{(z, \t)} w=0.
   \ee
   As mentioned in previous sections, the admissible crack distributions obtained
    for the Laplace equation will remain valid for
the bi-Laplace one \ef{bila}, \ef{Dir1}, having also other types
of singularities at the origin.

\subsection{Regularity via Hermitian spectral theory for a pencil}
 \label{S3}


\noindent  We now obtain a type of pencil operator needed to
tackle the problems under analysis in this paper. Also,
we shall perform our analysis on the basis of a non-self-adjoint
spectral pencil theory previously unknown and, probably, one of
the reasons these results could not be obtained before.
Indeed, eventually, we will put in charge a wider family of
harmonic polynomials, which is not that surprising.

\vspace{0.2cm}

\noindent\underline{Blow-up scaling}. 
Firstly, we perform  the same ``blow-up" scaling near the origin
$0$ for the bi-Laplace equation \eqref{resc1}, which was done
before for the Laplace equation \eqref{Lap1}. 

Remember via the rescaling \eqref{scal} we have transformed the singularity point at the origin 
to an asymptotic convergence when $\tau \to \infty$, as performed for the Laplace equation.   

Thus, thanks to 
operator \eqref{www.20} we get the rescaled one,
\be
\label{Bilap1}
\begin{split}
\D_{(z,\t)} \D_{(z, \t)} w & = {\mathrm e}^{2\tau}[ D_\tau^4+6
D_\tau^3+11D_\tau^2+6 D_\tau]w
\\ & + {\mathrm e}^{2\tau}[ 44zD_{z\tau}^2+24z D_{z\tau\tau}^3+10(1+3z^2)D_{zz\tau}^3]w +
\\ & + {\mathrm e}^{2\tau}[
4zD_{z\tau\tau\tau}^4+ 2(1+3z^2)D_{zz\tau\tau}^4+4z(1+z^2)D_{zzz\tau}]w
\\ & +
{\mathrm e}^{2\tau}[ (1+z^2)^2D_{z}^4 + 12z(1+z^2)D_{z}^3 +
12(1+3z^2)D_{z}^2 + 24 z D_{z}]w=0.
\end{split}
\ee
Therefore, we arrive at the
equation
 \be
 \label{wwpen}
 \begin{split}
  w_{\t\t\t\t} +6w_{\t\t\t}+11w_{\t\t} & + 6 w_\t + 44 z w_{z  \t}+ 24z   w_{z  \t\t}+ 10(1+3z^2)  w_{z z \t}+ 4z  w_{z  \t\t\t}
  \\ & +2 (1+3z^2) w_{z z \t\t}+ 4z(1+z^2)  w_{z z \t\t}= {\bf C}^* w,
 \end{split}
   \ee
 where the operator ${\bf C}^*$ stands for
 \[{\bf C}^* w \equiv -(1+z^2)^2 w_{z zzz} - 12z(1+z^2) w_{z  zz} - 12(1+3z^2) w_{z  z} - 24 z w_{z  }.\]
Now, as for the Laplace equation we are looking for solutions such that
 \be
  \label{zero}
   w(z,\tau) \to 0 \asA \t \to +\iy .
    \ee
    
    \vspace{0.2cm}

\noindent\underline{Pencil operator}. 
Again, thanks to Kondratiev's ``evolution" approach, we will show
that also for the bi-Laplace equation \eqref{bila}, with the
multiple crack condition \eqref{curv1} under consideration, we
need  polynomial eigenfunctions of certain pencil operators. To do
so, we write the solutions of \eqref{wwpen} in separate variables
 \[ w(z , \t)= {\mathrm e}^ {\l \t} \psi^*(z ) \whereA {\rm Re} \, \l
 <0 \,\,\, \mbox{by (\ref{zer1})},
 \]
  $\l$ stand for the eigenvalues of the adjoint operator ${\bf
 C}^*$,
and $\psi^*$ the corresponding eigenfunctions,
 arriving at an eigenvalue problem for a {\em polynomial $($quartic$)$ pencil}
of non self-adjoint operators of the form
 \be
 \label{BiPen.1}
 \begin{split}
 {\bf F}_\l^* \psi^* \equiv  & \{(\l^4+6\l^3+11\l^2+6\l)I+ 4 (\l^3+6\l^2+11\l) zD_z
 \\ & +2(1+3z^2)(\l^2+5\l)D_z^2+4\l(1+z^2)z D_z^3  - {\bf C}^*\} \psi^*=0.
 \end{split}
   \ee
   
   \vspace{0.2cm}

\noindent{\bf Remark.} The  fourth-order operator ${\bf C}^*$ is singular  at the
infinite points $z= \pm \infty$, so this is a singular pencil
eigenvalue problem. In this case, we also have that the operator
is not symmetric (similarly to the Laplace equation: see Remark\;\ref{re34}), since, for instance, the linear first-order
operator in \ef{BiPen.1},  $z D_z$, is not symmetric in $L^2$. 

One can 
see that introducing any weighted $L^2_\rho$ metric
does not help either. Indeed, a single weight function $\rho(z)$ is not
enough to arrange a symmetry balance. Thus, a symmetry feature is
not crucial at all for a functional setting to be used, though the
quality of particular functional spaces to be used remains
essential for eigenvalue analysis. In particular, the
analyticity properties/conditions obviously remain valid for
the bi-Laplace equation, so that, for finite-order zeros at 0,
harmonic polynomials must appear again.

\vspace{0.2cm}

\noindent\underline{Polynomial eigenfunctions and families of eigenvalues}. 
Similarly, as we obtained for the Laplace equation in
Proposition\;\ref{Pr.PolP}, we have that the eigenfunctions of the
\emph{adjoint pencil} \eqref{BiPen.1} are finite polynomials (cf.
the above analyticity demand). 

Moreover, we can state the following.
\begin{lemma}
The pencil operator \eqref{BiPen.1} admits four families of eigenfunctions 
\be
\label{eig47}
\{\psi_{l,1}^*(z)\},\quad \{\psi_{l,2}^*(z)\},\quad \{\psi_{l,3}^*(z)\},\quad \{\psi_{l,4}^*(z)\},
   \ee
   associated with four families of eigenvalues of the form
    \be
 \label{biei11}
 \begin{split} & \tex{
  \l_{l,1}=-l, \quad l=1,2,3,... \qquad \quad \l_{l,2}=-l-1, \quad
  l=0,1,2,3,...} \\ &  \tex{\l_{l,3}=-l-2, \quad
  l=0,1,2,3,... \andA  \l_{l,4}=-l-3, \quad
  l=0,1,2,3,... \,
  }
  \end{split}
  \ee
\end{lemma}
\begin{proof}
To find the
corresponding point spectrum of the pencil we just substitute the
$l$th-order polynomial eigenfunctions \eqref{eig00}
\be
\label{bieifun}
  \tex{
\psi_l^*(z)=z^l+b_{l-2}z^{l-2}+ b_{l-4}z^{l-4}+... =
\sum\limits_{k=l,l-2,...} b_k z^k, \quad (b_l=1),
  }
\ee
into (\ref{BiPen.1}) obtaining the following equation for the
eigenvalues $\l$:
  \be
  \label{bieig11}
  \begin{split}
 O(z^l): \;  \l_l^4 + 2(2l+3) \l_l^3 + (6l^2 +18 l +11)\l_l^2 & +(4 l^3+18l^2+22l+6)\l_l \\ & +l^4 +6 l^3+11l^2+6 l=0.
 \end{split}
   \ee
Subsequently, we solve this characteristic equation ascertaining the corresponding families of eigenvalues. Thus, taking into account that the negative eigenvalues
obtained for the quadratic pencil \eqref{Pen.1}
\[
 \tex{
  \l_l^+=-l, \quad l=1,2,3,... \andA \l_l^-=-l-1, \quad
  l=0,1,2,3,... \,,
  }
 \]
  are going to be solutions of the characteristic equation \eqref{bieig11}, we have that \eqref{bieig11} can be written by
  \[\tex{(\l_l +l)(\l_l +l+1) (\l_l^2 +(2l+5)\l_l +l^2+5l+6)=0.}\]
  Hence, we find four families of negative eigenvalues \eqref{biei11}. 
  \end{proof}
  
   Therefore, calculating  the {\em harmonic
 polynomials} as the corresponding eigenfunctions of the pencil,
 we arrive at.

  \begin{theorem}
  \label{Th.3}
 The fourth-order pencil \ef{BiPen.1} has four discrete spectra
 \ef{biei11} of real negative eigenvalues with the finite polynomial
 eigenfunctions given by \ef{bieifun}, where the expansion
 coefficients satisfy
 finite Kummer-type recursion   corresponding to the operator in
$(\ref{BiPen.1})$:
  \be
  \label{bico1}
   \begin{split}
   & \tex{
    b_{k+4}= -
    \frac{2\l_{l,i} (\l_{l,i}+5) +4\l_{l,i}  k+ 2 k(k-1)+12k + 12}{ (k+4)(k+3)} \,
      b_{k+2}} \\
       &
 \tex{
        - \frac{\l_{l,i}\left[ (\l_{l,i})^3
    +6(\l_{l,i})^2+ 11\l_{l,i}+6\right]+ 4\l_{l,i} \left[(\l_{l,i})^2+6\l_{l,i}+11\right]+6\l_{l,i}(\l_{l,i}+5)k(k-1)
   }{ (k+4)(k+3)(k+2)(k+1)} \,  b_{k},
   }
   \\ &   \tex{- \frac{ 4\l_{l,i} k(k-1)(k-2)+k(k-1)(k-2)(k-3) +12k(k-1)(k-2)+36 k(k-1)+24k
   }{ (k+4)(k+3)(k+2)(k+1)} \,  b_{k},}
   \end{split}
    \ee
    for $k\geq 4$, any $i=1,2,3,4$, and
    \begin{align*}
    & \tex{\l_{l,i} ((\l_{l,i})^3 +6(\l_{l,i})^2+ 11\l_{l,i}+6) b_0 + [4((\l_{l,i})^2+5\l_{l,i})+24] b_2+24 b_4=0,}\\ &
    \tex{\left[(\l_{l,i})^4 +10(\l_{l,i})^3 +17(\l_{l,i})^2 +17 \l_{l,i} +24\right] b_1+12\left[ (\l_{l,i})^2+7\l_{l,i} +12\right] b_3 +120 b_5=0,}\\ &
    \tex{ \left[ (\l_{l,i})^4 +10(\l_{l,i})^3 +47 (\l_{l,i})^2 +110 \l_{l,i} +120\right] b_2+24\left[ (\l_{l,i})^2+9\l_{l,i} +20\right]  b_4 + 360 b_6=0,}\\ &
    \tex{ \left[ (\l_{l,i})^4+10(\l_{l,i})^3 +71(\l_{l,i})^2 +254 \l_{l,i} +460\right] b_3 +240\left[ \l_{l,i}+5\right] b_5 +840b_7=0,}
    \end{align*}
  \end{theorem}

 \noi{\em Proof.}
    Similarly to the proof of Theorem\;\ref{Th.1} due to the previous Lemma, via \eqref{biei11}  the pencil \ef{BiPen.1} has four discrete spectra
 \ef{biei11} of real negative eigenvalues with four families of finite ($z$-re-written harmonic) polynomial
 eigenfunctions given by \eqref{eig47}, of the polynomial form \ef{bieifun}, and
associated with the four families of eigenvalues $\l_l^1$,
$\l_l^2$, $\l_l^3$, and $\l_l^{4}$, such that
\[
   \psi_{l,1}^*(z)\equiv \psi_{l,1}^*(z),\quad
\psi_{l,2}^*(z)\equiv \psi_{l-1,1}^*(z),\quad
\psi_{l,3}^*(z)\equiv \psi_{l-2,3}^*(z), \quad
\psi_{l,4}^*(z)\equiv \psi_{l-3,4}^*(z).
  \]
  Then, substituting
$\psi_l^*=\sum_{k\geq l} a_k z^k$, for any $l\geq 0$, into
(\ref{BiPen.1})
    we obtain that, for any $\l$,
\begin{align*}
  \tex{ \l(\l^3
  } & \tex{
  +6\l^2+ 11\l+6)\sum\limits_{k\geq 0}^l b_k z^{k}  + 4\l(\l^3+6\l^2 +11) z \sum\limits_{k\geq 1}^l  k b_k  z^{k-1}
   }  \\ & \tex{
 +2\l(\l+5) (1+3z^2) \sum\limits_{k\geq 2}^l k(k-1) b_k z^{k-2}
  +4\l z(1+z^2)\sum\limits_{k\geq 3}^l k(k-1)(k-2) b_k z^{k-2}
  }  \\ &  \tex{
 +  (1+z^2)^2 \sum\limits_{k\geq 4}^l k(k-1)(k-2)(k-3) b_k z^{k-4}   + 12z(1+z^2) \sum\limits_{k\geq 3}^l k(k-1)(k-2)(k-3) b_k z^{k-3}
 } \\  &  \tex{
 +12(1+3z^2) \sum\limits_{k\geq 2}^l k(k-1) b_k z^{k-2}+ 24 z   \sum\limits_{k\geq 1}^l k b_k  z^{k-1}=0,
 }
 \end{align*}
and, hence, rearranging terms 
\be
\label{bicoefla}
\begin{split}
 \tex{ \sum\limits_{k\geq 4}
} & \tex{
[\l(\l^3 +6\l^2+ 11\l+6)+ 4\l (\l^2+6\l+11)+6\l(\l+5)k(k-1)
 }
 \\ &  \tex{
  +4\l k(k-1)(k-2)+k(k-1)(k-2)(k-3) +12k(k-1)(k-2)+36 k(k-1)
 }
  \\   &    \tex{
  +24k]b_k z^k
 +
\sum\limits_{k\geq 4} [ 2\l (\l+5) (k+2)(k+1)+4\l (k+2)(k+1) k+ 2 (k+2)(k+1)k(k-1)
 }
 \\ &  \tex{
+12(k+2)(k+1)k + 12(k+2)(k+1) ]b_{k+2} z^k
 }
 \\  &  \tex{
 + \sum\limits_{k\geq 4} (k+4)(k+3)(k+2)(k+1) b_{k+4}z^k=0.
 }
\end{split}
\ee
Also, the first
four terms of the polynomial \eqref{bieifun} provide us with the
following equations for the first coefficients:
   \[\l (\l^3 +6\l^2+
11\l+6) b_0 + [4(\l^2+5\l)+24] b_2+24 b_4=0,\] 
\[(\l^4 +10\l^3
+17\l^2 +17 \l +24)b_1+12(\l^2+7\l +12)b_3 +120 b_5=0,\] 
\[(\l^4
+10\l^3 +47 \l^2 +110\l +120) b_2+24(\l^2+9\l +20) b_4 + 360
b_6=0,\] 
\[ (\l^4+10\l^3 +71\l^2 +254 \l +460)b_3 +240(\l+5) b_5
+840 b_7=0,\]
    proving the expression \eqref{bico1}. This completes the proof.$\qed$

    \ssk
    
\noi{\bf Remark.} \rm{ Again we can deduce that those coefficients
might have the expression
 \begin{align*}
   & \tex{
    b_{l-2n}= -
    \frac{N(l,\l_{l,i})
   }{ D(l,\l_{l,i})} \,  b_{l-2n+2}- \frac{(l-2n+4)(l-2n+3)(l-2n+2)(l-2n+1)
   }{ D(l,\l_{l,i})} \,  b_{l-2n+4},}\\ & \tex{
    n=1,2,...,[\frac l2], \quad b_l=1,\quad i=1,2,3,4,
    }
    \end{align*}
    where
    \[\tex{ N(l,\l_{l,i})= (l-2n+2)(l-2n+1)\big[2(l-2n)(l-2n+11)+12+2\l_{l,i}(\l_{l,i}+5)+4\l_{l,i} (l-2n)\big],}\]
    \begin{align*}
    D(l,\l_{l,i}) & = 24(l-2n)+36(l-2n)(l-2n-1)+12(l-2n)(l-2n-1) (l-2n-2)\\ & + (l-2n)(l-2n-1) (l-2n-2)(l-2n-3)
    \\ & + 4\l_{l,i} (l-2n)(l-2n-1) (l-2n-2)+ 6\l_{l,i} (\l_{l,i}+5)(l-2n)(l-2n-1)\\ & +
    4\l_{l,i} (l-2n)\big[(\l_{l,i})^2 +6\l_{l,i}+11\big] + \l_{l,i}\big[ (\l_{l,i})^3 +6(\l_{l,i})^2+11\l_{l,i} +6\big].
    \end{align*}
  }

\begin{remark} 
  Note that, even though, in this case, due to the discrete spectra, we again find certain relations for the families of eigenvalues
 \[
  \tex{\l_{l,4} =\l_{l,1}-3 =\l_{l-3,1}, \quad \l_{l,3} =\l_{l,1}-2 =\l_{l-2,1} \quad  \hbox{and}\quad \l_{l,2} =\l_{l,1}-1 =\l_{l-1,1}
 ,}
\]
 However, we find different polynomials (four different
families) depending on the considered eigenvalue. Indeed, by the
analyticity, those are {\em harmonic} ones but represented in a
different manner by using the rescaled variable $z$.
\end{remark}


\subsection{Nonexistence result for the bi-Laplace crack problem}


First, we
observe that our generalised polynomials \eqref{bieifun} are
harmonic polynomials, so that  these are also complete in any
reasonable weighted $L^2$ space. Therefore, again in this
situation, sufficient regular solutions of \eqref{wwpen},
\eqref{zero} should admit the corresponding eigenfunction
expansions over the polynomial families
\[\Phi^*=\{\psi_{l,1}^*,\psi_{l,2}^*, \psi_{l,3}^*,\psi_{l,4}^*
\},\]
 such that
 \be
 \label{exp112}
  \tex{
  w(z,\t)= \sum_{(k \ge l)} \eee^{-k \t} [C_k\psi_{k,1}^*(z)+ D_k \psi_{k-1,2}^*(z)+E_k \psi_{k-2,3}^*(z)+F_k \psi_{k-3,4}^*(z)],
  }
  \ee
 where four collections of expansion coefficients $\{C_k\}$, $\{D_k\}$, $\{E_k\}$ and $\{F_k\}$
  (which depend on boundary
 data on $\O$) take place and such that
 \[\psi_{l,1}^*(z) \equiv \psi_{l,1}^*(z),\quad \psi_{l,2}^*(z) \equiv \psi_{l-1,2}^*(z),\quad \psi_{l,3}^*(z) 
 \equiv \psi_{l-2,3}^*(z),\quad \psi_{l,4}^*(z) \equiv \psi_{l-3,4}^*(z).\]
 Thus, we state the following result:

\begin{theorem}
\label{Th.4}
Let the cracks $\G_1$,...,$\G_m$ in \ef{curv1} be
asymptotically given by $m$ different straight lines \ef{str1}.
Then, the following hold:
\begin{enumerate}
\item[{\rm (i)}] If
 all
$\{\a_k\}$ do not coincide with  all $m$ subsequent zeros of any
non-trivial linear combination
 \be
 \label{dm12}
C_l \psi^*_{l,1}(z)+D_l \psi^*_{l,2}(z)+E_l \psi^*_{l,3}(z)+F_l \psi^*_{l,4}(z), \quad \mbox{with} \quad
C_l^2+D_l^2+ E_l^2+F_l^2 \ne 0,
 \ee
 of the finite transformed harmonic
 polynomials $\psi_{l,1}^*(z)$, $\psi_{l,2}^*(z)$,
$\psi_{l,3}^*(z)$, and $\psi_{l,4}^*(z)$ defined by \ef{bieifun},
\ef{bico1} for any $l=m,m+1,...$ and arbitrary constants $C_l, \,
D_l,\, E_l,\, F_l \in \re$, then the multiple crack problem
\ef{bila} cannot have a solution for any boundary Dirichlet data
$g$, $h$ on $\O$.
\item[{\rm (ii)}] If, for some $l$, the distribution of zeros in {\rm
(i)} holds and a solution $u(x,y)$ exists,
 then
 \be
  \tex{
  |u(x,y)| = O(|x,y|^{l})    \asA (x,y) \to (0,0).
    }
  \ee
  \end{enumerate}

\end{theorem}

 \noi{\em Proof.} To prove Theorem\;\ref{Th.4} we follow a similar 
 argument as that performed for Theorem\;\ref{Th.2}. Indeed, we can also assure that those
 expansions \eqref{exp112} will converge in $L^2_\rho$,
with an appropriate exponentially decaying weight at infinity and
uniformly on compact subsets. 

Thus, the elliptic evolution while
approaching the origin actually occurs on compact, arbitrarily
large subsets for $z= \frac x{(-y)}$. Now, choosing $$C_k^2+ D_k^2
+E_k^2 +F_k^2 \neq 0,\quad (\hbox{the first non-zero expansion
coefficients}),$$ we arrive at the sharp asymptotics of the
solution
\be
  \label{st.22}
  w_l(y,\t)= C_l \eee^{-l \t}[C_k\psi_{k,1}^*(z)+ D_k \psi_{k-1,2}^*(z)+E_k \psi_{k-2,3}^*(z)+F_k \psi_{k-3,4}^*(z)] + O(\eee^{-(l+3) \t}),
    \ee
  as $ \t \to +\iy$. Hence, we have that the straight-line cracks \ef{str1} correspond to zeros of
the linear combination \ef{dm12},
\[C_k\psi_{k,1}^*(z)+ D_k \psi_{k-1,2}^*(z)+E_k \psi_{k-2,3}^*(z)+F_k \psi_{k-3,4}^*(z),\]
proving Theorem\;\ref{Th.4}.$\qed$

\begin{remark} 
Concerning the positive existence counterpart of our analysis, the
result is the same: multiple cracks at 0 can occur {\em iff} the
boundary data is taken from the expansion \ef{exp112}. This allows us 
to derive the co-dimension of this linear subspace of admissible
data. If the tip of the cracks is not fixed at the origin, then
the unity of all data (and an appropriate closure, if necessary)
should be taken in \ef{exp112} over all tip crack points $x_0 \in
\O$.
\end{remark}

\vspace{0.5cm}
\noindent{\bf Acknowledgments:} The authors wish to thank the anonymous referees
and the editor for their valuable comments and suggestions.

\vspace{0.5cm}

\end{document}